%% file: QEonLocSym.7a.tex
\documentclass[12pt]{amsart}
\usepackage[T1]{fontenc}
\usepackage{amssymb}
\usepackage{a4wide}
\usepackage{graphicx}
\usepackage{verbatim}
\usepackage{amsmath}
\usepackage{version}
\usepackage{hyperref}
 \theoremstyle{plain}    
 \newtheorem{thm}{Theorem}[section]
 \numberwithin{equation}{section} %% Comment out for sequentially-numbered
 \numberwithin{figure}{section} %% Comment out for sequentially-numbered
 \theoremstyle{plain}
 \newtheorem{prop}[thm]{Proposition} %%Delete [thm] to re-start numbering
 \newtheorem{lem}[thm]{Lemma} %%Delete [thm] to re-start numbering
 \theoremstyle{definition}
 \newtheorem{rem}[thm]{Remark} %%Delete [thm] to re-start numbering
 \newtheorem{definition}[thm]{Definition}
 \newtheorem{ex}[thm]{Example}

\input{def.7.tex}
\begin{document}

\title[Quantum limits on locally symmetric spaces]
{A Haar component for quantum limits on locally symmetric spaces}

   \author[N. Anantharaman]{Nalini Anantharaman}
   \author[L. Silberman]{Lior Silberman}

\address{Laboratoire de Math\'ematique, Universit\'e d'Orsay Paris XI, 91405 Orsay Cedex, France}
\email{Nalini.Anantharaman@math.u-psud.fr}
\address{Department of Mathematics, University of British Columbia,
  Vancouver\ \ BC\ \ V6T 1Z2, Canada}
\email{lior@math.ubc.ca}

\begin{abstract}
We prove lower bounds for the entropy of limit measures associated
to non-degenerate sequences of eigenfunctions on locally symmetric spaces
of non-positive curvature.
In the case of certain compact quotients of the space of positive definite
$n\times n$ matrices (any quotient for $n=3$, quotients associated to inner
forms in general), measure classification results then show that the limit
measures must have a Lebesgue component.  This is consistent with the
conjecture that the limit measures are absolutely continuous.
\end{abstract}

\thanks{N. Anantharaman wishes to acknowledge the support of Agence Nationale de la Recherche,
under the grants ANR-09-JCJC-0099-01 and ANR-07-BLAN-0361.}

\maketitle

 %  \abstract{We study the joint eigenfunctions of the algebra of translation-invariant differential operators, on a compact locally symmetric space of nonpositive curvature. We focus on the case of compact quotients of $SL(n, \R)/SO(n, \R)$. The quantum unique ergodicity conjecture predicts that eigenfunctions should become equidistributed asymptotically, as the spectral parameter goes to infinity. For $n=3$, we are able to show that any weak limit associated with a sequence of eigenfunctions has a Haar component of weight at least $\frac14$. For $n>3$, we obtain a similar result for cocompact quotients coming from lattices of ``inner type'' of $SL(n, \R)$. The method combines
 %entropy estimates (that constitutes an improvement of those obtained by Anantharaman-Nonnenmacher~\cite{AN07}) and measure classification results of Einsiedler-Katok, Einsiedler-Katok-Lindenstrauss, and Lindenstrauss-Weiss.}

%\part{ENTROPY ESTIMATES FOR EIGENFUNCTIONS ON LOCALLY SYMMETRIC SPACES}
\tableofcontents
\section{Introduction \label{s:intro}}
\subsection{Background and motivations}
The study of high-energy Laplacian eigenfunctions on negatively curved manifolds has progressed considerably in recent years. In the so-called ``arithmetic'' case, Elon Lindenstrauss has proved the Quantum Unique Ergodicity conjecture for Hecke eigenfunctions on congruence quotients of the hyperbolic plane~\cite{Linden06}.
In the ``general case'' (variable negative curvature, with no arithmetic structure), the first author has proved that semiclassical limits of eigenfunctions have positive Kolmogorov-Sinai entropy, in a joint work with St\'ephane Nonnenmacher~\cite{An, AN07, AKN07}. 
%This result has been extended by Gabriel Rivi\`ere~\cite{Riv09} to the case of nonpositively curved surfaces.

The two approaches are very different, but have in common the central role of the notion of entropy.
In Lindenstrauss' work, an entropy bound is obtained from arithmetic considerations~\cite{BLi03}, and then combined
with the measure rigidity phenomenon to prove Quantum Unique Ergodicity.

It is very natural to ask about a possible generalization of these results to locally symmetric spaces of higher rank and nonpositive curvature. In this case the Laplacian will be replaced by the entire algebra of translation-invariant differential operators, as proposed by Silberman and Venkatesh in~\cite{SilVen1}. A generalization of the entropic bound of~\cite{BLi03} has been worked out by these authors in the adelic case, and as a result they could prove a form of Arithmetic Quantum Unique Ergodicity in the case of the locally symmetric space $\Gamma\backslash SL_n(\IR)$, when $n$ is prime and $\Gamma$ is derived from
a division algebra over $\IQ$ \cite{SilVen2}. The goal of this paper is to generalize the ``non-arithmetic'' approach of~\cite{AN07, AKN07} in this context -- that is to say, prove an entropy bound without using the Hecke operators or other arithmetic techniques.  Doing so, we will not require
some of the assumptions used in~\cite{SilVen2}: we will work with an arbitrary
connected semisimple Lie group with finite center $G$, $\Gamma$ will be any {\em cocompact} lattice in $G$, and we will not
use the Hecke operators.  Combining the entropy bound with the measure classification results of~\cite{EK03, EKL06, LW01},
in the case of $G=SL_3(\IR)$, $\Gamma$ arbitrary, or $G=SL_n(\IR)$, $n$ arbitrary but $\Gamma$ derived from
a division algebra over $\IQ$, we will prove a weakened form
of Quantum Unique Ergodicity : any semiclassical measure has the Haar measure
as an ergodic component\footnote{Unfortunately, we are not able to extend the method to the case of $\Gamma=SL_n(\IZ)$, which is not cocompact -- unless we input the extra assumption that there is no escape of mass to infinity, or that the mass escapes very fast.}.

In addition to the intrinsic interest of locally symmetric spaces, there is yet another motivation to study these models. So far, the entropic bound of~\cite{AN07, AKN07} is not satisfactory for manifolds of variable negative curvature (\cite{An} proves that the entropy is positive, but without giving an explicit bound). Gabriel Rivi\`ere has been able to treat the case of surfaces~\cite{Riv09-1, Riv09-2}; he is even able to work in nonpositive curvature, but the case of higher dimensions remains open. The problem comes from the existence of several distinct Lyapunov exponents at each point. Locally symmetric spaces are an attempt to make some progress in this direction~: we will deal with flows that have distinct Lyapunov exponents, some of which may even vanish. Still, considerable simplifications arise from the fact that they are homogeneous spaces, and that the stable and unstable foliations are smooth.  It would be extremely interesting to extend the techniques of~\cite{AN07, AKN07, Riv09-1, Riv09-2} to systems that are not uniformly hyperbolic (euclidean billiards would be the ultimate goal). 

 Let $G$ be a connected semisimple Lie group with finite center, $K<G$ be a maximal
compact subgroup, $\Gamma<G$ a uniform lattice. We will work on the symmetric space $\bS=G/K$, the compact quotient\footnote{We do not assume that $\Gamma$ is torsion free. When speaking of smooth functions on $\bY$, we have in mind smooth functions on $\bS$ that are $\Gamma$-invariant.} $\bY=
\Gamma\backslash G/K$, and the homogeneous space $\bX=\Gamma\backslash G$.  We will endow $G$ with its Killing metric, yielding a $G$-invariant Riemannian metric on $G/K$, with nonpositive curvature.
%We will normalize the Haar measures $dx$ on $X$, $dk$ on $K$ and $dy$ on $\bY$ to have total mass $1$. We will denote by $e$ the identity element in $G$, and $o=eK$ the origin in $G/K$.

Call $\cD$ the algebra of $G$-invariant differential operators on $\bS$;
it follows from the structure of semisimple Lie algebras that
this algebra is commutative and finitely generated~\cite[Ch.\ II \S4.1, \S5.2]{Hel}. The number of generators,
to be denoted $r$, coincides with the real rank of $\bS$ (that is
the dimension of a maximal flat totally geodesic submanifold),
and, in a more algebraic fashion, with the dimension of $\fA$,
a maximal abelian semisimple subalgebra\footnote{We shall denote
$\fG$ the Lie algebra of $G$, $\fK$ the Lie algebra of $K$, and so on.} of $\fG$ orthogonal to $\fK$. More background and notations concerning Lie groups are given in Section \ref{s:lie}.

\begin{rem} The algebra $\cD$ always contains the Laplacian. If the symmetric space $\bS$ has rank $r=1$, then $\cD$ is generated by the Laplacian.
\end{rem} 

\begin{ex} The case $G=SO_o(d, 1)$ yields the $d$-dimensional hyperbolic space
$S=\IH^d$ (of rank $1$), already dealt with in~\cite{AN07, AKN07}.

We will focus on the example of $G=SL_n( \IR)$, $K=SO(n, \IR)$.
In that case, $\fG$ is the set of matrices with trace $0$, $\fK$ the antisymmetric matrices, and one can take
$\fA$ to be the set of diagonal matrices with trace $0$. The connected group generated by $\fA$ is denoted $A$, in this example it is the set
of diagonal matrices of determinant $1$ and with positive entries. The rank is $r=n-1$.
\end{ex}

We will be interested in $\Gamma$-invariant joint eigenfunctions of $\cD$; in other words, eigenfunctions of $\cD$ that go to the quotient $\Gamma\backslash G/K$. If we choose a set of generators of $\cD$, the collection of eigenvalues
can be represented as an element of $\IR^{r}$. We will recall in Section \ref{s:HC} that it is more natural to parametrize 
the eigenvalue by an element $\nu \in \fA_{\IC}^*$, the complexified dual of
$\fA$.  More precisely, $\nu\in\fA_{\IC}^*/W$ where $W$ is the Weyl group of $G$, a finite group given by $M'/M$ where $M'$ is the normalizer of $A$, and $M$ the centralizer of $A$, in $K$.  

\subsection{Semiclassical limit\label{s:semi}}
Silberman and Venkatesh suggested to study the $L^2$-normalized eigenfunctions $(\psi)$ in the limit
$\norm{\nu}\To +\infty$, as a variant of the very popular question of understanding high-energy
eigenfunctions of the Laplacian. 
The question of ``quantum ergodicity'' is to understand the asymptotic
behaviour of the family of probability measures
$d\bar\mu_\psi(y)=\abs{\psi(y)}^2 dy$ on $\bY=\Gamma\backslash G/K$.
They considered the case where
$\frac{\nu}{\norm{\nu}}$ has a limit $\nu_\infty\in  \fA_{\IC}^*/W$,
with the sequence $\nu$ satisfying a certain number of additional assumptions that we shall recall later.
For the moment, we just note that the real parts $\Re e(\nu)$ are uniformly bounded,
so that $\Re e(\nu_\infty)=0$ (\cite[Thm.\ 2.7 (3)]{SilVen1}).
We will denote $\Lambda_\infty=\Im m(\nu_\infty)=-i\nu_\infty$.

\subsection{Symplectic lift vs.\ representation-theoretic lift.}
The locally symmetric space $\bY$ should be thought of as the
\emph{configuration space} of our dynamical system.  To properly analyze
the dynamics it is necessary to move to an appropriate \emph{phase space}.
Once we \emph{lift} the eigenfunctions there, the measures become approximately
invariant under the dynamics and we can apply the tools of ergodic theory.
Two different kinds of lifts have been considered thus far:
the \emph{microlocal lift} (we also call it the {\em symplectic lift})
lifts the measure $\bar\mu_\psi$ to a distribution $\tilde \mu_\psi$ on the
cotangent bundle $T^*\bY=\Gamma\backslash T^*(G/K)$, taking advantage of its
symplectic structure.  This construction applies in great generality,
for example when $\bY$ is any compact Riemannian manifold.
The \emph{representation theoretic lift} used in
\cite{Wol01, Linden06, SilVen1, SilVen2, BuOlb}, specific to locally
symmetric spaces, lifts the measure $\bar\mu_\psi$ to a measure
$\mu_\psi$ defined on $\bX=\Gamma\backslash G$, taking advantage
of the homogeneous space structure of $G/K$.

The two lifts are very natural, and closely related. In our proofs we
will use a lot the symplectic point of view, as we will use the
Helgason-Fourier transform of $L^2$ functions, and interprete it
geometrically as a decomposition into lagrangian states.  But we will
also need to translate our results in terms of the representation
theoretic lift, in order to apply some measure classification results
from~\cite{EK03, EKL06}.

In the symplectic point of view, the dynamics is defined as follows.
On $T^*(G/K)$, consider the algebra $\cH$ of smooth $G$-invariant
Hamiltonians, that are polynomial in the fibers of the projection
$T^*(G/K)\To G/K$. This algebra is isomorphic to the algebra of
$W$-invariant polynomials on $\fA^*$
(consider the restriction on $\fA^*\subset T_o^* (G/K)$).
The structure theory of semisimple Lie algebras shows that $\cH$ is
isomorphic to a polynomial ring in $r$ generators.  Moreover, the elements
of $\cH$ commute under the Poisson bracket. Thus, we have on $T^*(G/K)$ a
family of $r$ independent commuting Hamiltonian flows $H_1,..., H_r$.
The Killing metric, seen as a function on $T^*(G/K)$, always belongs to $\cH$,
and its symplectic gradient generates the geodesic flow.
Of course, since all these flows are $G$-equivariant, they descend
to the quotient $T^*\bY$.

Joint energy layers of $\cH$ are naturally parametrized by elements
$\Lambda\in \fA^*/W$.
This is easy to explain geometrically: fix a point in $G/K$, say the
origin $o=eK$.  Consider the flat totally geodesic submanifold
$A.o\subset G/K$ going through $o$.  It is isometric to $\IR^r$,
and the cotangent space $T^*_o (A.o)$ is naturally isomorphic to $\fA^*$.
If $\cE\subset T^*(G/K)$ is a joint energy layer of $\cH$
(or equivalently a $G$-orbit in $T^*(G/K)$),
then there exists $\Lambda\in \fA^*$ such that
$\cE\cap T^*_o (A.o)=W.\Lambda$. See~\cite{Hil} for details.
We will denote $\cE_{\Lambda}$ the energy layer of parameter $\Lambda$.

In Section \ref{s:lift} we will use a quantization procedure to associate
to every $\Gamma$-invariant eigenfunction
$\psi$ a distribution $\tilde\mu_\psi$ on $T^*\bY$, called its microlocal lift.
This distribution projects to $\bar\mu_\psi$ on $\bY$.
This is a very standard construction, and so is the following theorem,
which is an avatar of propagation of singularities for solutions
of partial differential equations:

\begin{thm} \label{t:prop}
Assume that $\norm{\nu}\To +\infty$, and that
$\frac{\nu}{\norm{\nu}}$ has a limit $\nu_\infty$.
Denote $\Lambda_\infty=-i\nu_\infty\in \fA^*/W$.
Any limit (in the distribution sense) of the sequence $\tilde\mu_\psi$
is a probability measure on $T^*\bY$, carried by the energy layer
$\cE_{\Lambda_\infty}$, and invariant under the family of Hamiltonian
flows generated by $\cH$.
\end{thm}

In order to transport this statement to get an $A$-invariant measure
on $\Gamma\backslash G$, we must now make some assumptions on $\Lambda_\infty$. Silberman and Venkatesh assume that $\nu_\infty$ is a regular element
of $\fA^*_{\IC}$, in the sense that it is not fixed by any non-trivial
element of $W$, and they show that it implies $\Re e(\nu_n)=0$ for all but
a finite number of $\nu_n$s in the sequence.
The element $\nu_\infty$ being regular is, of course, equivalent to
$\Lambda_\infty$ being regular; and this is also equivalent to the energy layer $\Lambda_\infty$ being regular, in the sense that the differentials
$dH_1,...,dH_r$ are independent there~\cite{Hil}.

 There is a surjective map
\begin{eqnarray}\label{e:pi}
\pi: G/M\times \fA^* &\To&  T^*(G/K)\\
(gM, \lambda) &\mapsto & (gK, g.\lambda).
\end{eqnarray}Remember that $M$ is the centralizer of $A$ in $K$.
The image of $G/M\times\{\lambda\}$ under $\pi$ is the energy layer
$\cE_{\lambda}$. The map $\pi_\lambda : G/M\times\{\lambda\}\To \cE_\lambda$
is a diffeomorphism if and only if $\lambda$ is regular (otherwise $\pi_\lambda$ is not injective).
Under $\pi_\lambda^{-1}$, the action of the Hamiltonian flow $\Phi^t_H$ generated by $H\in\cH$ on $\cE_\lambda$
is conjugate to 
$$gM\mapsto g \exp(t\,dH(\lambda))M.$$
The same statements hold after quotienting
on the left by $\Gamma$. Since $H$ is a function on $\fA^*$, the differential $dH(\lambda)$ is an element of $\fA$.
Denoting $R(e^{tX})$ the one--parameter
	flow on $G/M$ generated by $X\in \fA$ (acting by multiplication on the right), we can rephrase this by writing
	$$\pi\!\circ R(e^{t dH(\lambda)})=\Phi_H^t\circ\pi \qquad \mbox{ on }\cE_\lambda.$$

If $\lambda$ is regular, the elements $dH(\lambda)$ can be shown to
span $\fA$ as $H$ varies over $\cH$. Otherwise, we have~\cite{Hil}
\begin{equation}\label{e:span} 
\{dH(\lambda), H\in\cH\}=\{X\in \fA, \forall \alpha\in \Delta,\left( \left\la \alpha, \lambda\right\ra=0 \Longrightarrow
\alpha(X)=0\right)\},
\end{equation}
where $\Delta\subset\fA^*$ is the set of roots.

Thus, Theorem \ref{t:prop} may be rephrased as follows:
\begin{thm} Assume $\Lambda_\infty$ is regular. Then any limit (in the distribution sense) of the sequence $\tilde\mu_\psi$
yields a probability measure on $\Gamma\backslash G/M  $, invariant under the
right action of $A$ by multiplication.\end{thm}

This theorem was proved in~\cite[Thm.\ 1.6 (3)]{SilVen1} using the
representation-theoretic lift;  the equivariance of that lift shows that
the construction is compatible with the Hecke operators on
$\Gamma\backslash G$.  It is also shown there that the symplectic
lift $\tilde\mu_\psi$ and the representation theoretic lift $\mu_\psi$
have the same asymptotic behaviour as $\nu$ tends to infinity, and if we
identify $\cE_{\Lambda_\infty}\subset \Gamma\backslash T^* (G/K)$ with
$\Gamma\backslash G/M$.

\begin{definition} We will call any limit point of the sequence
$\tilde\mu_\psi$ (or $\mu_\psi$) a \emph{semiclassical measure} in the direction $\Lambda_\infty$.
\end{definition}
Semiclassical measures in a regular direction are, equivalently, positive measures on $T^*(\Gamma\backslash G/K)$ (carried by a regular energy layer),
positive measures on $\Gamma\backslash G/M$, or positive measures on $\Gamma\backslash G $ (which are $M$-invariant).

\subsection{Entropy bounds}
Our main result is a non-trivial lower bound on the entropy of semiclassical measures.
We fix $H\in\cH$, and we consider the corresponding Hamiltonian flow $\Phi_H^t$ on $\cE_{\Lambda_{\infty}}$,
which has Lyapunov exponents
$$-\chi_J(H) \leq \cdots \leq -\chi_1(H)
    \leq 0
  \leq \chi_1(H) \leq \cdots \leq \chi_J(H).$$
 In addition, the Lyapunov exponent $0$ appears trivially with multiplicity $r$, as a consequence
 of the existence of $r$ integrals of motion. The dimension of $\cE_{\Lambda_{\infty}}$ is $r+2J$.
 The integer $J$, the rank $r$
and the dimension $d$ of $G/K$ are related by $d=J+r$.
In general, the Lyapunov exponents are measurable functions on the
phase space, but here, because of the homogeneous structure,
the Lyapunov exponents are constants.  

In the following theorem we will denote $\chi_{\max}(H) = \chi_J(H)$,
the largest Lyapunov exponent. We denote 
$h_{KS}(\mu, H)$ the Kolmogorov-Sinai entropy of a $(\Phi_H^t)$-invariant
probability measure $\mu$. We recall the Ruelle-Pesin inequality,
$$h_{KS}(\mu, H)\leq \sum_j \chi_j(H),$$
which holds for any $(\Phi_H^t)$-invariant probability measure $\mu$.

\begin{thm} (Symplectic version)\label{t:symplectic} Let $\mu$ be a semiclassical measure
in the direction $\Lambda_\infty$. Assume that $\Lambda_\infty$ is regular.

For $H\in\cH$, we consider the corresponding Hamiltonian flow
$\Phi_H^t$ on $\cE_{\Lambda_{\infty}}$.  Then
\begin{equation}\label{e:proved1}
  h_{KS}(\mu, H) \geq
  \sum_{j:\chi_j(H)\geq\frac{{\chi_{\max}}(H)}{2}}
      \left(\chi_j(H)-\frac{\chi_{\max}(H)}{2}\right).
\end{equation}
\end{thm}
Continuing with the assumption that $\Lambda_\infty$ is regular, we can transport the theorem
to $\Gamma\backslash G/M$. If we fix a 1-parameter subgroup $(e^{tX})$ of $A$
(with $X\in\fA$), it is well known that the (non trivial) Lyapunov exponents of the flow
$(e^{tX})$ acting on $X/M$ are the real numbers $(\alpha(X))$, where
$\alpha\in \fA^*$ run over the set of roots $\Delta$
(see Section \ref{s:lie} for background related to Lie groups).
If $\alpha$ is a root then so is $-\alpha$ (one of the two will be called
\emph{positive}, the other \emph{negative}).  The notion of positivity
is explained in detail later.  For now it suffices to note that we may
assume that $\alpha(X)\geq 0$ for positive roots $\alpha$.  We write
$\alpha_{\max}(X)$ for $\max_\alpha \alpha(X)$ (this is the largest Lyapunov
exponent of the associated Hamiltonian flow).
Each root occurs with \emph{multiplicity} $m_\alpha$, which must be
taken into account in the statements below (the corresponding Lyapunov
exponent $\alpha(X)$ would be counted repeatedly, $m_\alpha$ times).

\begin{thm}\label{t:main2} (Group-theoretic version)  
Let $\mu$ be a semiclassical measure in the direction $\Lambda_\infty$.
Assume that $\Lambda_\infty$ is regular.

Let $(e^{tX})$ ($X\in\fA$) be a one parameter subgroup of $A$ such that
$\alpha(X)\geq 0$ for all positive roots $\alpha$.

Let $h_{KS}(\mu, X)$ be the entropy of $\mu$ with respect to the flow $(e^{tX})$. Then 
\begin{equation}\label{e:proved}
h_{KS}(\mu, X) \geq
  \sum_{\alpha:\alpha(X)\geq\frac{\alpha_{\max}(X)}{2}}
   m_\alpha \left( \alpha(X)-\frac{\alpha_{\max}(X)}{2}\right).
\end{equation}
\end{thm}

Our lower bound is positive for all non-zero $X$, in fact greater than
$\frac{\alpha_{\max}(X)}{2}$.
  In~\cite{An, AN07}, the first author and S. Nonnenmacher had conjectured the following stronger bound
$$h_{KS}(\mu, H)\geq\frac12 \sum_{j} \chi_j(H) $$
or equivalently
\begin{equation}\label{e:optim}
h_{KS}(\mu, X)\geq \frac12\sum_{\alpha>0} m_\alpha \cdot \alpha(X) .
\end{equation}
We are still unable to prove it, except in one case: when all the positive
Lyapunov exponents are equal to each other, so that formula \eqref{e:proved}
reduces to \eqref{e:optim}.
One case is that of hyperbolic $d$-space ($G = SO(d,1)$) alluded to above.
Another, the main focus of the present paper, is of the
``extremely irregular'' elements of the torus in $G=SL_n(\IR)$.
These are the elements conjugate under the Weyl group to
$$X=\diag(n-1, -1, ...,-1).$$

\subsection{Application: towards Quantum Unique Ergodicity on locally 
symmetric spaces\label{s:QUE}}
%pb de notation :X

In Section \ref{s:rigid} we combine our entropy bounds with
measure classification results.  Let $n\geq 3$, $G=SL_n(\IR)$,
$\Gamma <G$ a cocompact lattice.
Let $\mu$ be a semiclassical measure on $\Gamma \backslash G$ in the regular direction
$\Lambda_\infty$.

The measure $\mu$ can be written uniquely as a sum of an absolutely continuous
measure and a singular measure (with respect to Lebesgue or Haar measure).
Since $\mu$ is invariant under the action of $A$, the same holds for both components.
Because the Haar measure is known to be ergodic for the action of $A$, the absolutely continuous part of $\mu$ is, in fact,
proportional to Haar measure.  We call
this the \emph{Haar component} of $\mu$.  Its total mass is the \emph{weight}
of this component.
%(this terminology comes from considering the ergodic
%decomposition of $\mu$ \wrt the $A$-action).
 
\begin{thm}\label{t:coroSL3}
Let $n=3$.  Then $\mu$ has a Haar component of weight $\geq \frac14$.
\end{thm}

\begin{thm}\label{t:coroSL4}
Let $n=4$.  Then either $\mu$ has a Haar component, or each ergodic component
is the Haar measure on a closed orbit of the group
$\left( \begin{array}{cccc} 
* & * &*& 0 \\
* & * &*& 0 \\
* & * &*& 0 \\
 0& 0 &0& * \\
\end{array} \right)$
(or one of its 4 images under the Weyl group), and the components invariant
by each of these 4 subgroups have total weight $\frac14$.
\end{thm}

In fact, the result is slightly stronger: if some ``extremely irregular''
element acts on $\mu$ with entropy strictly larger than half of its entropy
\wrt Haar measure, then there is a Haar component.

It does not seem to be possible to push this technique beyond $SL_4$.
The problem is that there are large subgroups (in the style of those
occuring in Theorem \ref{t:coroSL4}) whose closed orbits support measures
of large entropy.  For particular lattices, however, these large subgroups
do not have closed orbits, so the only possible non-Haar components have
small entropy and cannot account for all the entropy.  For co-compact
lattices this occurs, for example, when $\Gamma$ is the set of elements
of reduced norm $1$ of an order in a central division algebra over $\IQ$,
or more generally for any lattice commensurable with one obtained this way
(we say that $\Gamma$ is \emph{associated} to the division algebra).
Such lattices are said to be of ``inner type'' since they correspond to
inner forms of $SL_n$ over $\IQ$ (there also exist non-uniform lattices
of inner type, corresponding to central simple $\IQ$-algebras which are
not division algebras).  For a brief description of the construction and
references see Section \ref{s:rigid}.

\begin{thm}\label{t:coroinner}
For $n\geq 3$ let $\Gamma < \SL_n(\IR)$ be a lattice associated to a 
division algebra over $\IQ$, and let $\mu$ be a semiclassical measure
on $\Gamma \backslash \SL_n(\IR)$ in a regular direction.
Then $\mu$ has a Haar component of weight
$\geq \frac{\frac{n+1}{2}-t}{n-t} > 0$
where $t$ is the largest proper divisor of $n$.
\end{thm}

It is not surprising that strongest implication is for $n$ prime
(so that there are few intermediate algebraic measures).
Indeed, setting $t=1$ we find $w_\Delta \geq \frac12$ in that case.
However for $n$ prime Silberman-Venkatesh~\cite{SilVen2} show that
the semiclassical measures associated to {\em Hecke eigenfunctions} are
equal to Haar measure.  The main impact of Theorem \ref{t:coroinner}
is thus when the $n$ is composite, where previous methods only showed
that semiclassical measures are convex combinations of algebraic
measures but could not establish that Haar measure occurs in the combination.

\begin{rem} We compare here our result with that of~\cite{SilVen2}.
That paper studies the case of lattices in $G=PGL_n(\IR)$ associated to
division algebras of prime degree $n$ and joint eigenfunctions
of $\cD$ and of the Hecke operators.  It is then shown that
\emph{any ergodic component of a semiclassical measure $\mu$}
has positive entropy; it follows that $\mu$ must be the Haar measure.
Our result is neither stronger nor weaker:

\begin{itemize}

\item We cannot prove that all ergodic components of $\mu$ have positive
entropy, only that the total entropy of $\mu$ is positive. Hence, we
are not able to exclude components of zero entropy;

\item On the other hand, our lower bound on the total entropy
($1/2$ of the maximal entropy) is explicit and quite strong. This allows
to detect the presence of a Haar component in a variety of cases;

\item In particular, for $n=3$ we do not need any assumption on the cocompact
lattice $\Gamma$; and for $\Gamma$ associated to a division algebra,
our result holds for all $n$.

\item The Hecke-operator method applies more naturally to adelic
quotients $\IG(\IQ) \backslash \IG(\IA) / K_\infty K_\textrm{f}$.  When
$G$ is a form of $SL_n$ there is no distinction, but when $G = PGL_n$
the adelic quotients are typically \emph{disjoint unions} of quotients
$\Gamma \backslash G$.  Even when the quotient is compact, $G$-invariance
of the limit measure does not show that all components have the same
proportion of the mass.  Our result applies to each connected component
separately.

\item We do not assume that our eigenfunctions are also eigenfunctions
of the Hecke operators: this means that multiplicity of eigenvalues
is not an issue in this work.

\item The methods of Silberman-Venkatesh apply to non-cocompact lattices
as well.

\end{itemize}
\end{rem}

\subsection{Hyperbolic dispersive estimate}
	The proof of Theorem \ref{t:symplectic} (and \ref{t:main2}) follows the main ideas of~\cite{AN07}, with a major difference which lies in an improvement of the ``hyperbolic dispersive estimate''~: \cite[Thm.\ 1.3.3]{An} and \cite[Thm.\ 2.7]{AN07}. If we applied directly the result of~\cite{AN07},
	we would get
	$$h_{KS}(\mu, H)\geq \sum_{k }\left(\chi_k(H)-\frac{\chi_{\max}(H)}2\right).$$ This inequality is often trivial (the right-hand term being negative)
	whereas in \eqref{e:proved1} we managed to get rid of the negative terms $\left(\chi_k(H)-\frac{\chi_{\max}(H)}2\right)$.
	 
Since the ``hyperbolic dispersive estimate'' has an intrinsic interest, and is the core of this paper, we state it here as one of our main results. We fix a quantization procedure, set at scale $\hbar=\norm{\nu}^{-1}$, that associates to any reasonable function $a$ on $T^*\bY$ an operator $\Op_\hbar(a)$ on $L^2(\bY)$. An explicit construction is given in Section \ref{s:lift}. In particular, it is useful to know that  $\Op_\hbar$ can be defined so that, if $H\in\cH$ is real valued, $\Op_\hbar(H)$ is a self-adjoint operator belonging to $\cD$.  More explicitly, $\Op_\hbar(H)$ is defined so that $\Op_\hbar(H)\psi_\nu=H(-i\hbar \nu)\psi_\nu$ for any $\cD$-eigenfunction $\psi_\nu$, with spectral parameter $\nu$ (hence the choice of the normalisation $\hbar=\norm{\nu}^{-1}$).  

	Let $(P_k)_{k=1,\ldots,K}$ be a family of smooth real functions on $\bY$, such that 
	\begin{equation}\label{e:partition}
	\forall x\in \bY,\qquad \sum_{k=1}^K P_k^2(x)= 1\,\,.
	\end{equation}
	We assume that the diameter of the supports of the functions $P_k$ is small enough.
	We will also denote $P_k$
	the operator of multiplication by $P_k(x)$ on the Hilbert space $L^2(\bY)$. 
	
		We denote $U^t=\exp(i\hbar^{-1} t\Op_\hbar(H))$ the propagator of the ``Schr\"odinger equation'' generated by the Hamiltonian $H$. This is a unitary Fourier Integral
	Operator associated with the classical Hamiltonian flow $\Phi_H^{-t}$.  The $\hbar$-dependence of $U$ will be implicit in our notations. We fix a small discrete time step $\eta$.

	Throughout the paper we will use the notation $\widehat A(t)=U^{-t\eta}\widehat A U^{t\eta}$ for the quantum evolution at time $t\eta$
	of an operator $\widehat A$.	  
 	For each integer $T\in\IN$ and any 
	sequence of labels ${\bom}=(\om_{-T},\cdots, \om_{-1},\om_0, \cdots \om_{T-1})$, 
	$\om_i\in [1,K]$ (we say that the sequence ${\bom}$ is of {\em length} $|{\bom}|=2T$),
	we define the operators
	\begin{equation} 
	\begin{split}
	%P_{{\bom}}&=P_{\omega_n} U P_{\omega_{n-1}}\ldots U P_{\omega_0}\\
	P_{{\bom}}=P_{\omega_{T-1}}(T-1)P_{\omega_{T-2}}(T-2)\ldots P_{\omega_0}
	P_{\om_{-1}}(-1)\ldots P_{\omega_{-T}}(-T)
\,\,.%\label{e:tP_bep}
	\end{split}
	\end{equation}

	We fix a smooth, compactly supported function $\chi$ on $T^*\bY$, supported
	in a tubular neighbourhood of size $\eps$ of the energy layer $\cE_{\Lambda_\infty}$ (which is assumed to be regular); and we define
	\begin{equation}\label{e:P_bep}
	\begin{split}
	%P_{{\bom}}&=P_{\omega_n} U P_{\omega_{n-1}}\ldots U P_{\omega_0}\\
	P_{{\bom}}^\chi=P_{\omega_{T-1}}(T-1)P_{\omega_{T-2}}(T-2)\ldots P^{1/2}_{\omega_0}
	\Op(\chi)P^{1/2}_{\omega_0}
	P_{\om_{-1}}(-1)\ldots P_{\omega_{-T}}(-T)
\,\,.%\label{e:tP_bep}
	\end{split}
	\end{equation}
The operator $P_{{\bom}}^\chi$ should be thought of as $P_{{\bom}}$ restricted to a spectral
window around the energy layer $\cE_{\Lambda_{\infty}}.$

	%%%%%%%%%%%%%%%%%%%%%%%%%%%%%%%%
	 \begin{thm} \label{t:main} Fix $H\in\cH$, and a time step $\eta$, small enough.
	  Let $\cK>0$ be fixed, arbitrary.  Let $\chi \in \cC^\infty(T^*\bY)$, supported
	in a tubular neighbourhood of size $\eps$ of the regular energy layer $\cE_{\Lambda_\infty}$. Assume that $\eps$, as well as the diameters of the supports of each $P_k$, are small enough. 

	Then, there exists $\hbar_\cK>0$ such that, for all $\hbar\in(0, \hbar_\cK)$, for $T=\lfloor\frac{\cK|\log\hbar|}\eta\rfloor$, and for every sequence ${\bom}$ of length $T$,
	\begin{equation}\label{e:mainineq}
	\norm{P_{\bom}^\chi}
	 \leq C\,\,\hbar^{-c\eps} \prod_{ k,\, \chi_k(H)\geq \frac1{2\cK} } \frac{e^{-T\eta\,\chi_k(H)}}{\hbar^{1/2}}\,\,
	\end{equation}
	where the $\chi_k(H)$ denote the Lyapunov exponents of $\Phi_H^t$ on the energy layer $\cE_{\Lambda_\infty}$.
	The constant $C$ does not depend on $\cK$ nor on $H$, whereas $c$ does.
	 \end{thm}
	%%%%%%%%%%%%%%%%%%%%%%%%%%%%%%%%

	 The method used in \cite{AN07} only yielded the upper bound:
	\begin{equation}\label{e:notsogood}
	\norm{P_{\bom}^\chi}
	 \leq C\,\,\hbar^{-c\eps} \prod_{ k } \frac{e^{-T\eta\,\chi_k(H)}}{\hbar^{1/2}}\,\,
	\end{equation}
	This is clearly not optimal when $\Phi_H^t$ has some neutral, or slowly expanding directions. For instance, if $H=0$
	then $\Phi_H^t=I$ has only neutral directions. In this case, \eqref{e:notsogood} reads
	\begin{equation}  \norm{P_{\bom}^\chi}
	 \leq C\,\,\hbar^{-\frac{d}2-c\eps}
	  ,
	\end{equation}
	where $d$ is the dimension of $\bY$,
	which is obviously much worse (for any $T$) than the trivial bound
	\begin{equation} \label{e:trivial}  \norm{P_{\bom}^\chi}
	 \leq 1.
	\end{equation}

	On the other hand, if some of the $\chi_k(H)$ are (strictly) positive, then \eqref{e:notsogood} is much better
	than the trivial bound \eqref{e:trivial}, for very large $T\eta$. The bound given by Theorem \ref{t:main} interpolates between the two, for $T\eta\sim \cK |\log\hbar|$. 
	
	The proof of the hyperbolic dispersion estimates is quite technical, and occupies Sections \ref{s:lift}, \ref{s:WKB}, \ref{s:CS}. It uses a version of the pseudodifferential calculus adapted
	to the geometry of locally symmetric spaces, based on Helgason's version of the Fourier transform for this spaces, and inspired by the work of Zelditch in the case of $G=SL(2, \R)$ \cite{Zel86}.
	We point out the fact that an alternative proof of Theorem \ref{t:main} is given in \cite{A-add},
	based on more conventional Fourier analysis. The reader might prefer to read \cite{A-add}
	instead of Sections \ref{s:lift}, \ref{s:WKB}, \ref{s:CS}, however we feel that the two techniques have an interest of their own.
	
We will not repeat here the argument that leads from Theorem \ref{t:main} to the entropy bound Theorem \ref{t:symplectic}; it would be an exact repetition of the argument given in \cite[\S2]{AN07}. Let us just 
make one comment~: in this argument, we are limited to $\cK=\frac{1}{\chi_{\max}(H)}$ (the time $T_E= \frac{|\log\hbar|}{\chi_{\max}(H)}$ is sometimes called the Ehrenfest time for the Hamiltonian $H$, and corresponds to the time where the approximation of the quantum flow $U^t$ by the classical flow $\Phi_H^t$ breaks down). This means that we eventually keep the Lyapunov exponents such that $\chi_k(H)\geq \frac{\chi_{\max}(H)}2,$ and explains why this restriction appears in \eqref{e:proved1}.

%{\bf Acknowledgements:}  
 %This work was partially supported by the grant ANR-05-JCJC-0107-01.
%I am also grateful to the Miller Institute for Basic Research in Science, University of California Berkeley, for supporting my work in the spring 2009.

\section{Background and notation regarding semisimple Lie groups\label{s:lie}}
Our terminology follows Knapp \cite{Knapp86}.
\subsection{Structure}
Let $G$ denote a non-compact connected simple Lie
group with finite center\footnote{If $G$ is semisimple our discussion remains valid, but one can even do something finer, as remarked in~\cite[\S 5.1]{SilVen1}. After decomposing $\fG$ into simple factors $\oplus \fG^{(j)}$, and assuming that the Cartan involution, the subalgebra $\fA$, etc. are compatible with this decomposition, one can decompose the spectral parameter $\nu$ into its components $\nu^{(j)}\in \fA^{(j)*}$. Instead of assuming that $\norm{\nu}\To +\infty$ and
$\frac{\nu}{\norm{\nu}}$ has a regular limit $\nu_\infty$, one can assume the same independently for each component $\nu^{(j)}$. This means that we do not have to assume that all the norms $\norm{\nu^{(j)}}$ go to infinity at the same speed.}.
We choose a Cartan involution $\Theta$ for $G$, and let $K < G$ be the
$\Theta$-fixed maximal compact subgroup. Let $\fG = Lie(G)$, and let $\theta$
denote the differential of $\Theta$, giving the Cartan
decomposition $\fG= \fK\oplus \fP$ with $\fK = Lie(K)$. 
Let $\bS = G/K$ be the symmetric space, with
$o=eK\in \bS$ the point with stabilizer $K$. We fix a $G$-invariant metric on
$G/K$: observe that the tangent space at the point $o$ is naturally
identified with $\fP$, and endow it with the Killing form.
For a lattice $\Gamma < G$ we write $\bX =\Gamma\backslash G$ and
$\bY = \Gamma\backslash G/K$, the latter being a
locally symmetric space of non-positive curvature. In this paper, we shall always assume that $\bX$ and $\bY$ are compact.
%We normalize the Haar
%measures $dx$ on $\bX$, $dk$ on $K$ and $dy$ on $\bY$ to have total mass $1$
%(here $dy$ is the pushforward of $dx$ under the projection from $\bX$ to $\bY$,
%given by averaging with respect to $dk$). 

Fix now a maximal abelian subalgebra $\fA\subset\fP$.

We denote by $\fA_{\IC}$ the complexification $\fA\otimes \IC$.
We denote by $\fA^*$ (resp. $\fA_{\IC}^*$) the real dual
(resp. the complex dual) of $\fA$.
For $\nu\in \fA_{\IC}^*$, we define $\Re e(\nu), \Im m(\nu) \in \fA^*$
to be the real and imaginary parts of $\nu$, respectively.
For $\alpha\in\fA^*$,  set $\fG_\alpha  = \{X\in\fG, \forall H \in \fA : ad(H)X =  \alpha(H)X \}$,
$\Delta=\Delta (\fA: \fG) = \{  \alpha\in\fA^*  \setminus \{0\},  \fG_\alpha\not = \{0\}\}$ and call the latter the (restricted)
roots of $\fG$ with respect to $\fA$. The subalgebra $\fG_0$ is  $\theta$-invariant, and hence $\fG_0=
(\fG_0\cap\fP)\oplus (\fG_0\cap\fK)$. By the maximality of $\fA$ in $\fP$, we must then have $\fG_0 = \fA\oplus \fM$
where $\fM = Z_{\fK}(\fA)$, the centralizer of $\fA$ in $\fK$.

The Killing form of $\fG$ induces a standard inner product $\left\la ., .\right\ra$ on $\fP$, and by duality on $\fP^*$.
By restriction we get an inner product on $\fA^*$  with respect to
which  $\Delta(\fA: \fG) \subset\fA^*$ is a root system. The associated Weyl group, generated
by the root reflections $s_\alpha$, will be denoted $W=W(\fA: \fG)$. This group is
also canonically isomorphic to $N_K(\fA)/Z_K(\fA)$. In what follows we will represent any element $w$ of the Weyl group by a representative in
$N_K(\fA)\subset K$ (taking care to only make statements that do not depend
on the choice of a representative), and the action of $w\in W(\fA:\fG)$ on
$\fA$ or $\fA^*$ will be given by the adjoint representation $\Ad(w)$.
The fixed-point set of any $s_\alpha$ is a hyperplane in $\fA^*$,
called a wall. The connected components of the complement of the union
of the walls are cones, called the (open) Weyl chambers.
A subset $\Pi\subset\Delta(\fA: \fG)$ will be called a system of simple roots
if every root can be uniquely expressed as an integral
combination of elements of $\Pi$ with either all coefficients non-negative or
all coefficients non-positive.
For a simple system $\Pi$, the open cone
$ C_\Pi = \{ \nu\in\fA^*, \forall \alpha\in\Pi : \left\la\nu, \alpha\right\ra>0\}$ is an (open) Weyl chamber.
The closure of an open chamber will be called a closed chamber;
we will denote in particular $\overline{ C_\Pi}  =
\{  \nu\in\fA^*, \forall \alpha\in\Pi : \left\la\nu, \alpha\right\ra\geq 0\}$.  The
Weyl group acts simply transitively on the chambers and simple systems.
The action of $W(\fA: \fG)$ on $\fA^*$ extends in the complex-linear
way to an action on $\fA_{\IC}^*$ preserving $i\fA^*\subset\fA_{\IC}^*$,
and we call an element $\nu\in\fA_{\IC}^*$ regular if it is fixed by
no non-trivial element of $W(\fA: \fG)$.  Since $-C_\Pi \subset \fA^*$ is a chamber,
there is a unique $\wl\in W(\fA: \fG)$, called the ``long element'',
such that $\Ad(\wl).C_\Pi=- C_\Pi$.  Note that $\wl^2 \C_\Pi = C_\Pi$
and hence $\wl^2=e$.  Also, $\wl$ depends on the choice of $\Pi$ but
we suppress this from the notation.

Fixing a simple system $\Pi$ we get a notion of positivity.
We will denote by $\Delta^+$ the set of positive roots,
by $\Delta^-=-\Delta^+$ the set of negative roots.
We use $\rho= \frac12\sum_{\alpha>0} (\dim \fG_\alpha)\alpha\in\fA^*$
to denote half the sum of the positive roots.
For $\fN = \oplus_{\alpha >0}\fG_\alpha$ and
$\lienb = \Theta\fN = \oplus_{\alpha <0}\fG_\alpha$ we have
$\fG = \lien \oplus \fA \oplus \fM \oplus \lienb $.
Note that $\lienb =\Ad(\wl).\fN$.
We also have (``Iwasawa decomposition'') $\fG = \fN \oplus \fA \oplus \fK$.
We can therefore uniquely write every $X\in\fG$ in the form
$X = X_\fN + X_\fA + X_\fK$. We also write $H_0(X)$ for $X_\fA$.

Let $N, A, \overline{N} < G$ be the connected subgroups corresponding to the subalgebras
$\fN, \fA, \lienb \subset\fG$ respectively, and let $M = Z_K(\fA)$.
Then $\fM = Lie(M)$, though $M$ is not
necessarily connected. Moreover $P_0 = NAM$ is a minimal parabolic subgroup of
$G$, with the map $N \times A \times M \To P_0$ being a diffeomorphism.
The map $N \times A \times K \To G$ is a (surjective) diffeomorphism
(Iwasawa decomposition), so for $g\in G$ there exists a unique $H_0(g)\in\fA$
such that $g = n \exp(H_0(g))k$ for some $n\in N$,
$k\in K$. The map $H_0 : G \To \fA$ is continuous; restricted to $A$, it is the
inverse of the exponential map.

We will use the $G$-equivariant identification
between $G/M$ and $G/K\times G/P_0$, given by $gM\mapsto (gK, gP_0)$.
The quotient $G/P_0$ can also be identified with $K/M$. 

Starting from $H_0$ we define a ``Busemann function'' $B$ on
$G/K\times G/P_0\sim G/M$:
\begin{equation}\label{e:buse} B(gK, g_1P_0)=H_0(k^{-1}g), \end{equation}
where $k$ is the $K$-part in the $KAN$ decomposition of $g_1$ (if $g_1$ is defined modulo $P_0$, then $k$ is defined modulo $M$).
Equivalently, if $gM\in G/M$, we have $B(gM)=a$, where $g=kna$ is the $KNA$ decomposition of $g$ (if $g$ is defined modulo $M$, then $a$ is uniquely defined and $k$ is defined modulo $M$).

In $G/K$, a ``flat'' is a maximal flat totally geodesic submanifold.
Every flat is of the form $\{gaK, a\in A\}$
for some $g\in G$. The space of flats can be naturally identified with $G/MA$,
or with an open dense subset of $G/P_0\times G/\bPz$, via the $G$-equivariant map
$$gMA\mapsto (gP_0, g\bPz)$$
where $\bPz=MA\bN=\wl P_0 \wl^{-1}.$
We will also use the following injective map from $G/MA$ into
$G/P_0\times G/ P_0$, $$gMA\mapsto (gP_0, g \wl P_0).$$
Its image is an open dense subset of $G/P_0\times G/ P_0$, namely $\{(g_1 P_0, g_2 P_0),
g^{-1}_2 g_1\in P_0 \wl P_0\}$. Finally we recall the Bruhat decomposition
$G=\sqcup_{w\in W(\fA: \fG)} P_0 w P_0$, with $P_0 \wl P_0$ being an open dense subset (the ``big cell'').

\subsection{The universal enveloping algebra; Harish-Chandra isomorphisms\label{s:HC}}
We analyze the structure of $\cD$ by comparing it with other algebras of
differential operators.
For a Lie algebra $\fS$ we write $\fS_\C$ for its complexification
$\fS\otimes_\IR \IC$.  In particular, $\fG_\C$ is a complex semisimple
Lie algebra. We fix a maximal abelian subalgebra $\fB\subset\fM$ and let
$\fH=\fA\oplus\fB$. Then $\fH_\C$ is a Cartan subalgebra of $\fG_\C$,
with an associated root system $\Delta(\fH_\C : \fG_\C)$ satisfying
$\Delta(\fA: \fG) = \{\alpha_{| \fA}\}_{\alpha\in\Delta(\fH_\C : \fG_\C) }\setminus\{0\}$.

If $\fS_\C$ is a complex Lie algebra, we denote by $U(\fS_\C)$ its
universal enveloping algebra; $U(\fG_\C)$ is isomorphic to the algebra
of left-$G$-invariant differential operators on $G$ with complex coefficients
\cite{Gode}.

There is an isomorphism, called the Harish-Chandra isomorphism, between the algebra $\cD$ of $G$-invariant
differential operators on $G/K$ and the algebra $\cD_W(A)$ of
$A$- and $W$-invariant differential operators on $A\sim\IR^r$.
The latter is obviously isomorphic to $U(\fA_\IC)^W$,
the subalgebra of $U(\fA_\IC)$ formed of $W$-invariant elements. Since $\fA_\C$ is abelian, $U(\fA_\IC)$ is can be identified to the space of
polynomial functions on $\fA^*$ with complex coefficients.

The Harish-Chandra isomorphism $\Gamma: \cD\To \cD_W(A)$ can be
realized in a geometric way as follows \cite[Cor.\ II.5.19]{Hel}.
Consider the flat subspace $A.o\subset G/K$, naturally identified
with $A$. Fixing $D\in\cD$, let $\Delta_N(D)$ be the translation-invariant
differential operator on $A$ (that is, an element of $U(\fA)$) given by
$$[\Delta_N(D)f](a)=D\tilde f(a.o),$$
for $a\in A$, $f\in C^\infty(A.o)$, and where $\tilde f$ stands with the unique $N$-invariant function on $G/K$ that coincides with $f$ on $A.o$.
Then, we define
$$\Gamma : D\mapsto e^{-\rho}\!\circ \Delta_N(D)\!\circ e^{\rho},$$
remembering that $\rho$ is half the sum of positive roots and thus
can be seen as a function on $A$.  Note that
$$e^{-\rho}\!\circ \Delta_N(D)\!\circ e^{\rho}=\tau_\rho.\Delta_N(D),$$
where $\tau_\rho$ is the automorphism of $U(\fA)$ defined by putting $\tau_{\rho}(X)=X+\rho(X)$ for every $X\in\fA$.

In what follows, we denote by $\cZ(\fG_\C)$ the center of $U(\fG_\C)$. Thus, $\cZ(\fG_\C)$ is the algebra of
$G$-bi-invariant operators.   Differentiating the action of $G$ on $\bS$
gives a map $\cZ(\fG_\IC)\to\cD$.
For the next lemma we shall compare the isomorphism $\Gamma$ with 
an isomorphism $\omega_{HC}: \cZ(\fG_{\IC}) \To
U(\fH_{\IC})^{W(\fH_{\IC}:\fG_{\IC})}$, also called the Harish-Chandra
isomorphism\footnote{This is the isomorphism denoted by $\gamma_{HC}$ in \cite{SilVen1}, and defined by $\gamma_{HC}(z)=\tau_{\rho_\fH}{\rm pr}(z)$, where ${\rm pr}(z)\in U(\fH_{\IC})$ is such that $z-{\rm pr}(z)\in U(\fN_{\IC})U(\fA_{\IC})+U(\fG_{\IC})\fK_{\IC}$.}.

\begin{lem}\label{l:SV} Assume that the restriction from
$\fH_{\C}$ to $\fA$ induces a surjection from
$U(\fH_{\IC})^{W(\fH_{\IC}:\fG_{\IC})}$  to $U(\fA_\IC)^W$ (thought of as
functions on the respective linear spaces).

Let $D\in \cD$, of degree $\bd$. Then there exists $Z\in\cZ(\fG_\C)$
such that $Z$ and $D$ coincide on (right-)$K$-invariant functions, and such that
$$Z-\tau_{-\rho}\Gamma(D)\in U(\fN_{\IC})U(\fA_{\IC})^{\bd -2}+U(\fG_{\IC})\fK_{\IC}.$$
\end{lem}

\begin{rem} The assumption is automatically satisfied when $G$ is split.
It is also satisfied when $G/K$ is a {\em classical} symmetric space,
that is when $G$ is a classical group~\cite[p.\ 341]{Hel}.
In fact the lemma itself is Proposition II.5.32 of~\cite{Hel},
with the difference of degree between $Z$ and $\tau_{-\rho}\Gamma(D)$ made
precise.
\end{rem}

\begin{proof}
Let $D\in\cD$ be of degree $\bd$, so that $\Gamma(D)\in U(\fA_\IC)^W$ is a
polynomial of degree $\leq \bd$.  By assumption, we can extend 
$\Gamma(D)$ to an element of $U(\fH_{\IC})^{W(\fH_{\IC}:\fG_{\IC})}$.
Consider $Z_1=\omega_{HC}^{-1}\Gamma(D).$ It is shown in
\cite[Cor.\ 4.4]{SilVen1} that
$$Z_1-\tau_{-\rho}\Gamma(D)\in U(\fN_{\IC})U(\fA_{\IC})^{\bd-2}+U(\fG_{\IC})\fK_{\IC}.$$
It is not completely clear that $Z_1$ and $D$ coincide on $K$-invariant
functions, but the above formula shows that $\Gamma(Z_1)-\Gamma(D)$ is of
degree $\leq \bd-2$, and hence that $Z_1-D$ has degree at most $\bd-2$.

By descending induction on the degree of $\Gamma(Z)-\Gamma(D)$, we see that we can thus construct $Z\in \cZ(\fG_{\IC})$
such that
$$Z-\tau_{-\rho}\Gamma(D)\in U(\fN_{\IC})U(\fA_{\IC})^{\bd-2}+U(\fG_{\IC})\fK_{\IC}$$ and such that
$\Gamma(Z)-\Gamma(D)=0$ (which precisely means that $Z$ and $D$
coincide on right-$K$-invariant functions).
\end{proof}

\subsection{The Helgason-Fourier transform}
For any $\theta\in G/P_0$, $\nu\in\fA^*_{\IC}$, the function
$$e_{\nu, \theta}: x\in G/K \mapsto e^{(\rho+\nu)B(x, \theta)}$$
is a joint eigenfunction of $\cD$,
and one can verify easily (for instance in the case $\theta=eM$) that
$$ De_{\nu, \theta}=[\Gamma(D)](\nu)e_{\nu, \theta},$$
for every $D\in\cD$. Here we have seen $\Gamma(D)$ as a $W$-invariant
polynomial on $\fA^*_{\IC}.$

In fact for any joint eigenfunction $\psi$ of $\cD$ there exists
$\nu\in\fA^*_{\IC}$ such that $$D\psi=[\Gamma(D)](\nu)\psi$$
for every $D\in\cD$ \cite[Ch.\ II Thm.\ 5.18, Ch.\ III Lem.\ 3.11]{Hel}.
The parameter $\nu$ is called the ``spectral parameter'' of $\psi$;
it is uniquely determined up to the action of $W$.

The Helgason--Fourier transform gives the spectral decomposition of a function $u\in C_c^\infty(\bS)$ on the ``basis''  $(e_{\nu, \theta})$ of eigenfunctions of $\cD$. It is defined as
\begin{equation}\label{e:classicalHF}
\widetilde u(\lambda, \theta)=\int_{\bS} u(x)e_{-i\lambda, \theta}(x) dx,
\end{equation}
($\lambda\in\fA^*, \theta\in G/P_0$).  It has an inversion formula:
\begin{equation*}
u(x)=\int_{\theta\in G/P_0, \lambda\in \overline{ C_\Pi}}\widetilde u(\lambda, \theta)e_{i\lambda, \theta}(x)  d\theta |c(\lambda)|^{-2}d\lambda.
\end{equation*}
Here $d\theta$ denotes the normalized $K$-invariant measure on $G/P_0\sim K/M$.
The function $c$ is the so-called Harish-Chandra function, given
by the Gindikin-Karpelevic formula~\cite[Thm.\ 6.14, p.\ 447]{Hel}.
 
The Plancherel formula reads
$$\norm{u}^2_{L^2(\bS)}=\int_{\theta\in G/P_0, \lambda\in \overline{ C_\Pi}} |\widetilde u(\lambda, \theta)|^2 d\theta |c(\lambda)|^{-2}d\lambda.$$

\begin{rem} \label{r:D}For $D\in \cD$, $D$ acts on $u$ by
\begin{equation*}
Du(x)=\int_{\theta\in G/P_0, \lambda\in \overline{ C_\Pi}}[\Gamma(D)](i\lambda)\widetilde u(\lambda, \theta)e_{i\lambda, \theta}(x)  d\theta |c(\lambda)|^{-2}d\lambda
\end{equation*}
\end{rem}

%\begin{rem} \label{r:quasiinvariance}
%vraiment utile ? \spadesuit
%For later purposes we will need the following fact.
%If we fix $\lambda\in\ars$, the function $(x, \theta)\in G/K\times G/P_0\mapsto e_{i\lambda, \theta}(x)$ defines a function on $G/M$, that we call $F_\lambda$. It is given explicitely by 
%\begin{equation}\label{e:Flambda}F_\lambda: gM \in G/M \mapsto e^{(\rho+i\lambda)\log a}\end{equation}
%where $g=kna$ is the $KNA$ decomposition of $g$. 
%We note again that, if $g$ is defined modulo $M$, $a$ is uniquely defined, and $k$ is defined modulo $M$. If $h\in G$, we have
 %$$ F_\lambda(hg)= F_\lambda(h k) F_\lambda(g). $$
 
%Let now $D\in  U(\fN_\IC)U(\fA_\IC)$~: we can let it act as a differential operator on $F_\lambda$, if we lift $F_\lambda$ to a function on $G$. It is clear that the function
%$g\mapsto F_\lambda(g)^{-1} DF_\lambda(g)$ is a constant (depending on $D$ and $\lambda$).
%This will be used in Lemma \ref{l:crucial}.
%\end{rem} 
 
\section{Quantization and pseudodifferential operators\label{s:lift}}
In this section we develop a pseudodifferential calculus for $\bS$, inspired by
the work of Zelditch \cite{Zel86}. We do not push the analysis as far as in \cite{Zel86} (a more detailed analysis is done in Michael Schr\"oder's thesis \cite{SchDiss}). For us, the most
important feature of this quantization is that it is based on the Helgason-Fourier transform, in other words, on the spectral decomposition of the algebra $\cD$.

\subsection{Semiclassical Helgason transform}
We now introduce a parameter $\hbar$.  In the sequel it will tend to $0$
at the same speed as $\norm{\nu}^{-1}$; the reader may identify the two. The parameter will be assumed to go to infinity in the conditions of \S \ref{s:semi}, the limit $\nu_\infty$ assumed to be regular.

From now on we rescale the parameter space $\ars$ of the Helgason--Fourier
transform by $\hbar$.  We define the semiclassical Fourier transform,
$\widehat u_\hbar(\lambda, \theta)=\widetilde u(\hbar^{-1}\lambda, \theta).$
Thus, for $u\in C_c^\infty(\bS)$, we rewrite equation
\eqref{e:classicalHF} as:
$$\widehat u_\hbar(\lambda, \theta)=\int_{\bS} u(x)e_{-i\hbar^{-1}\lambda,\theta}(x) dx$$
($\lambda\in \overline{ C_\Pi}, \theta\in G/P_0$). 
The inversion formula now reads
$$u(x)=\int_{\theta\in G/P_0, \lambda\in \overline{ C_\Pi}} \widehat u_\hbar(\lambda, \theta)e_{i\hbar^{-1}\lambda, \theta}(x)  d\theta |c_\hbar(\lambda)|^{-2}d\lambda,$$
with the  ``semiclassical Harish-Chandra $c$-function'', 
$$|c_\hbar(\lambda)|^{-2}=  \hbar^{-r}|c(\hbar^{-1}\lambda)|^{-2}.$$
\begin{rem}\label{r:GK}
By the Gindikin-Karpelevic formula, we have
$$|c(\hbar^{-1}\lambda)|^{-2}\asymp \hbar^{-\dim \fN}$$
uniformly for $\lambda$ in a compact subset of $ C_\Pi$, and thus
$$|c_\hbar(\lambda)|^{-2}\asymp \hbar^{-d}$$
where $d=\dim\fA+\dim\fN=\dim(G/K)$.
\end{rem}
We also adjust the Plancherel formula to
$$\norm{u}^2_{L^2(\bS)}=\int |\widehat u_\hbar(\lambda, \theta)|^2 d\theta |c_\hbar(\lambda)|^{-2}d\lambda.$$

In the sequel we will always use the semiclassical Fourier transform, and will in general denote
$\widehat u$ instead of $\widehat u_\hbar$.

\subsection{Pseudodifferential calculus on $\bY$.}\label{s:PDO}
%%%%%%%%%%%%%%%%%%%%%%%%%%%%%%%%%%%%%%%%%%%%%%%%%%

We identify the functions on
the quotient $\bY=\Gamma\backslash G/K$  (respectively $ T^* \bY $) with the
$\Gamma$--invariant functions on $\bS=G/K$ (resp. $ T^*(G/K) $). If $\Gamma$ has torsion, we shall use ``smooth function on $\bY$'' to mean a $\Gamma$-invariant smooth function on $\bS$.
For a compactly supported function $\chi$ on $\bS$,
we denote $\Pi_\Gamma\chi(x)=\sum_\gamma \chi(\gamma.x)$.  This sum is finite
for any $x\in \bS$, and hence defines a function on $\bY$.

On $\bS$, we fix once and for all a positive, smooth and compactly supported
function $\phi$ such that $\sum_{\gamma\in\Gamma}\phi(\gamma.x)\equiv 1$.
We call such a function a ``smooth fundamental cutoff'' or a ``smooth
fundamental domain''.  Here we have used the assumption that $\bY$ is compact.
We also introduce $\tilde\phi\in C_c^\infty(\bS)$ which is identically $1$ on the support of $\phi$.
We note that for any $D\in\cD$ and for any smooth $\Gamma$-invariant
$u$ on $\bS$ we have
\begin{equation}\label{e:consist}
\Pi_\Gamma \left(\tilde \phi D\left(\phi u\right)\right)
   = \Pi_\Gamma D\left(\phi u\right)
   = D \Pi_\Gamma \phi u=Du.
\end{equation}

The analogue of left-quantization on $\IR^n$ in our setting associates
to a function $a$ on $ G/K\times G/P_0\times C_\Pi $
the operator which acts on $u\in C_c^\infty(G/K)$ by
\bequ\label{e:Weyl}
\Op^L_\hbar(a)\,\,u(x)=\int_{\theta\in G/P_0, \lambda\in \overline{ C_\Pi}}  
a\!\left(x, \theta, \lambda\right)\widehat u(\lambda, \theta)
e_{i\hbar^{-1}\lambda, \theta}(x)  d\theta |c_\hbar(\lambda)|^{-2}d\lambda \,\, .
\eequ
A similar formula was introduced by Zelditch in \cite{Zel86} (with $\hbar=1$) in the case $G=SL(2, \IR)$;  it is shown there that $a\mapsto \Op^L_\hbar(a)$ is $G$-equivariant.
The operator $\Op^L_\hbar(a)$ can be defined if $a$ belongs to a nice class of functions (possibly depending on
$\hbar$). If $a$ is smooth enough and has reasonable growth, it will be a pseudodifferential operator.
We give the regularity assumptions on $a$ below. In any case, we shall always require $a$ to be  
of the form $b\circ \pi$, where $b$ is a symbol on $T^* (G/K)$ and $\pi$ was defined in \eqref{e:pi}; besides, we will assume that $b$ is supported away from the singular $G$-orbits in $T^* (G/K)$ (which means that $a$ is supported away from the walls in $C_\Pi$). This allows to identify $a$ in a natural way with a function defined on (a subset of) $T^* (G/K)$.

Let us define {\it symbols of order $m$} on $T^* (G/K)$ (independent of $\hbar$) in the usual fashion~:
\begin{multline*} S^m (G/K):= \big\{ a \in C^\infty (T^* (G/K))/
\\
\mbox{ for every compact }F\subset G/K, \mbox{ for every }\alpha, \beta,\mbox{ there exists }C  \mbox{ such that }\\
|D_z^\alpha D_\xi ^\beta  a (x,\xi))|
\leq C(1+ |\xi|)^{m-|\beta|}\mbox{ for all }(x, \xi)\in T^*(G/K), x\in F  \big\} .\end{multline*}
%For instance, this class contains functions which are homogeneous in a neighbourhood of infinity.
%We denote $\Sigma ^{-\infty}= \cap _{m\in\Z} \Sigma ^m $ --- this class contains the smooth compactly supported functions $C_o^\infty (U \times \R^d)$.

We also define {\it semiclassical symbols of order $m$ and degree $l$} --- thus called  
because they depend on a parameter $\hbar$ :
\begin{equation}\label{e:symbols}S^{m,l}(G/K)=\{  a_\hbar(x, \xi) = \hbar^l\sum _{j=0}^\infty \hbar^j  a_j(x,\xi),~ a_j \in S^{m-j} \}.\end{equation}
This means that $a_\hbar( x, \xi )$ has an asymptotic expansion
in powers of $\hbar$, in the sense that
\[ a - \hbar^l\sum _{j=0}^{N-1} \hbar^j  a_j \in \hbar^{l+N}S^{m-N} \]
for all $N$, uniformly in $\hbar$.
In this context, we denote $S^{-\infty, +\infty}= \cap _{m\geq 0} S^{-m,m} $.

\begin{rem}As indicated above, we define symbols on $G/K\times G/P_0\times C_\Pi$ by transporting the standard definition on $T^* (G/K)$ through the map $\pi$ \eqref{e:pi}.
We will exclusively consider the case where $a$ vanishes outside a fixed neighbourhood of the singular $G$-orbits in $T^*(G/K)$. In other words, $a$ can be identified (through \eqref{e:pi}) with a function on $G/K\times G/P_0\times C_\Pi$, that vanishes in a neighbourhood of $G/K\times G/P_0\times \partial C_\Pi$.
 Defining a good pseudodifferential calculus using formula \eqref{e:Weyl} for symbols supported near the walls of  $C_\Pi$ raises delicate issues about the behaviour of the $c$-function near the walls, and we do not address this problem here. This is one among several reasons why we assume that $\Lambda_\infty$ is regular in our main theorem.
\end{rem}

%for instance the space of symbols
%\begin{equation}\label{e:S^mk}
%S^{m,k} \defeq
%\set{a=a_\hbar\in C_c^\infty( G/K\times G/P_0\times C_\Pi ),\ 
%|\partial_x^\alpha\partial_{(\theta, \lambda)}^\beta a |\leq 
%C_{\alpha,\beta } \hbar^{-k}\,\,\left\la\lambda\right\ra^{m-|\beta|}}.
%\end{equation}

We now project this construction down to functions on $\bY$, which we identify with
$\Gamma$-invariant functions on $\bS$. Here we do not follow Zelditch, who defined the action of 
$\Op_\hbar(a)$ on $\Gamma$-invariant functions in a global manner, using the Helgason-Fourier
decomposition of such functions. We continue to work locally, which is sufficient for our purposes.

For us, the quantization of
$a\in S^{m,k}\cap C^\infty(T^* \bY)$ (supported away from singular $G$-orbits) is defined to act on $u\in C^\infty(\bY)$ by:
\begin{equation}\label{e:Op}
\Op_\hbar(a)\,u=\Pi_\Gamma \tilde\phi
\Op_\hbar^L(a)\phi u \in C^\infty(\bY) .
\end{equation}
Note that \eqref{e:consist} and Remark \ref{r:D} imply that
$\Op_\hbar(H)=\Gamma^{-1}[H(-i\hbar\bullet)]$
for $H\in\cH$.

The image of $S^{m,k}$ by this quantization will be denoted $\Psi^{m,k}(\bY)$. 
This quantization procedure depends on the fundamental cutoff $\phi$ and on
$\tilde\phi$.  However, this dependence only appears at 
second order in $\hbar$.
%and the principal symbol map 
%$\sigma:\Psi^{m,k}(\bY)\to S^{m,k}/S^{m,k-1}$ is intrinsically defined.
The space $\Psi^{m,k}(\bY)$ itself is perfectly well defined modulo
$\Psi^{-\infty,+\infty}(\bY)=\cap_{k', m'} \Psi^{m',k'}(\bY)$. Moreover, it coincides with the more usual definition of pseudodifferential operators, defined using the euclidean Fourier transform in local coordinates\footnote{This could be checked by testing the action of $\Op_\hbar(a)$ on a local plane wave of the form $\phi(x)e^{\frac{i\xi.x}\hbar}$ in local euclidean coordinates. One then uses the stationary phase method and the facts that the complex phase of $e_{\hbar^{-1}\lambda, \theta}$ is $\hbar^{-1}\lambda B(x, \theta)$, and that the covector $(x, d_x\lambda B(x, \theta))\in
T^*_x(G/K)$ corresponds precisely to $(x, \theta, \lambda)$ under the identification \eqref{e:pi}.}.

% We actually need to consider symbols more general than \eqref{e:S^mk}.
%Following~\cite{DS99}, for any $0\leq \delta <1/2$ we introduce the symbol class 
%\bequ\label{e:symbol-eps}
%S_\delta^{m,k}\defeq\set{a\in C_c^\infty( G/K\times G/P_0\times C_\Pi ),\ |\partial_x^\alpha\partial_{(\theta, \lambda)}^\beta a |\leq 
%C_{\alpha,\beta }\,\, \hbar^{-k-\delta|\alpha+\beta |}\,\,\left\la\lambda\right\ra^{m-| \beta|}}\,\,.
%\eequ
%We can use standard results about pseudodifferential operators. For instance, the quantization of any $a\in S^{0,0}(G/K)$ leads
%to a bounded operator from $L^2_{comp}(G/K)$ to $L^2_{loc}(G/K)$ (where  $L^2_{comp}(G/K)$ 
%denotes $L^2$ functions with compact support, and $L^2_{loc}(G/K)$ denotes locally $L^2$ functions).
%If $a$ is in addition $\Gamma$-invariant, it follows that $\Op_\hbar(a)$ (defined by \ref{e:Op})
%is bounded on
%$L^2(\bY)$ (the norm being bounded uniformly in $\hbar$). 
 
%%%%%%%%%%%%%%%%%%%%%%%%%%%%%%%%%%%%%%%%%%%%%%%
\subsection{Action of $\Op_\hbar(H)$ on WKB states}
%%%%%%%%%%%%%%%%%%%%%%%%%%%%%%%%%%%%%%%%%%%%%%% 

Fix a Hamiltonian $H\in \cH$.
 
The letter $H$ will stand for several different objects which are
canonically related: a function $H$ on $T^*(G/K)$,
a $W$-invariant polynomial function on $\fA^*$, and an element of $U(\fA)^W$.
As such, we can also let $H$ act as a left-$G$-invariant differential
operator\footnote{We have also introduced the differential operator
$\Op_1(H)=\Gamma^{-1}[H(-i\bullet)]$ acting on $G/K$.
These are {\em not} the same objects, but \cite[Cor.\ 4.4]{SilVen1} relates the two.}
on $G$ or $G/M$.

%We note the relation
%$$\Op_1(H).e_{i\lambda, \theta}(x)=H(-i\bullet).F_\lambda(x, \theta),$$
%where $F_\lambda$ is the function defined in \eqref{e:Flambda}. On the left hand side
%$\Op_1(H)$ is acting on functions on $G/K$ whereas on the right $H(-i\bullet)$ is acting on $G/M$.
 
In the following lemma, all functions on $G/K$ and $G/M$ are lifted to functions on $G$, and in that sense we can apply to them any differential operator on $G$. If $b$ is a function defined on
$G/M=G/K\times G/P_0$, and $\theta$ is an element of $G/P_0$, we denote $b_\theta$ the function defined on $G/K$ by $b_\theta(x)=b(x, \theta)$.

\begin{lem} \label{l:crucial}
Let $H\in \cH$ be of degree $\bar d$, and let $b$ be a smooth function on $G/M $. Fix $\lambda\in\fA^*$.
Then, there exist $D_k \in U(\fN_\IC)U(\fA_\IC)$ of degree $\leq k$ (depending on $\lambda$ and on $H$) such that
for any $\theta\in G/P_0$, for any $x\in G/K$,
$$\Op_\hbar(H) [b_\theta.e_{i\hbar^{-1}\lambda, \theta }] (x)= 
\left( H(\lambda)b(x, \theta) -i\hbar [dH(\lambda).b](x, \theta) +\sum_{k=2}^{\bar d}\hbar^k D_k b (x, \theta)\right)e_{i\hbar^{-1}\lambda, \theta }(x).$$
On the right $H$ is seen as a function on $\fA^*$, so its differential
$dH(\lambda)$ is an element of $\fA$, and it
acts as a differential operator of order $1$ on $G/M$. 
 Each operator $D_k$ actually defines a differential operator on
	$G/M$.
	\end{lem}

	\begin{proof}
	By linearity, it is enough to treat the case where $H\in U(\fA)^W$ is homogeneous
	of degree $\bar d$. In this case, we have
	$$\Op_\hbar(H)=\hbar^{\bar d }\Op_1(H) =\hbar^{\bar d}\Gamma^{-1}[H(-i\bullet)].$$
	Consider the operator $Z$ related to $D=\Op_{1}(H)$ by Lemma \ref{l:SV}. We have
	$$\Op_1(H) [b_\theta.e_{i\hbar^{-1}\lambda, \theta }] (x)=Z[b_\theta.e_{i\hbar^{-1}\lambda, \theta }] (x).$$

	% Recall the isomorphism $G/M \To G/K\times G/P_0 $,
	%given by $gM\mapsto (gK, gP_0)$. The inverse map 
	%$G/K\times G/P_0 \To G/M$ goes as follows. Given $\theta=g_1P_0$, any $gK\in G/K$ can be represented uniquely as $gK=g_1 na K$. The element $(gK, g_1P_0)\in G/K\times G/P_0 $ is identified with $g_1 naM\in G/M$. 

	%Any function on $G/K$ can be pulled back to a function on $G/M$, and
	%the action of $Z$ on functions on $G/M$ preserves the subspace
	%of right $K$-invariant functions.

In what follows we consider the point $(x, \theta)\in G/K\times G/P_0$. We choose a representative of $\theta$ in $K$ ($\theta$ is then defined modulo $M$, but the calculations do not depend on the choice of this representative). We write $x=\theta na K$. This means that $(x, \theta)$ represents the point $\theta naM\in G/M$. All functions on $G/K$ and $G/M$ are lifted to functions on $G$, and in that sense we can apply to them any differential operator on $G$.

	By Lemma \ref{l:SV}, we have  
	\begin{multline*}Z[b_\theta.e_{i\hbar^{-1}\lambda, \theta }](x)= Z[b_\theta.e_{i\hbar^{-1}\lambda, \theta }](\theta na )=\tau_{-\rho}H(-i\bullet).[b_\theta.e_{i\hbar^{-1}\lambda, \theta }](\theta na)+
	D[ b_\theta.e_{i\hbar^{-1}\lambda, \theta } ](\theta na)
	\end{multline*}
	where $D\in U(\fN_{\IC})U(\fA_{\IC})^{\bd -2}$.

	Because of the identity
	$$e_{i\hbar^{-1}\lambda, \theta } (\theta nag)=e^{(\rho+i\hbar^{-1}\lambda)B(\theta na)} 
	e^{(\rho+i\hbar^{-1}\lambda)H_0(g)} ,$$
	(valid for any $g\in NA$) we see that, for any $D\in U(\fN_{\IC})U(\fA_{\IC})$, the term $D[e_{i\hbar^{-1}\lambda, \theta } ](\theta na)$ is of the form $C e_{i\hbar^{-1}\lambda, \theta } (\theta na)$, where the constant $C$ depends on $D$ and $\hbar^{-1}\lambda$. This constant $C$ is in fact polynomial in $\hbar^{-1}\lambda$.
	
	This results in an expression~:
	\begin{multline*}Z[b_\theta.e_{i\hbar^{-1}\lambda, \theta }](x)= Z[b_\theta.e_{i\hbar^{-1}\lambda, \theta }](\theta na )=\tau_{-\rho}H(-i\bullet).[b_\theta.e_{i\hbar^{-1}\lambda, \theta }](\theta na)\\+\left[
	\sum_{k=0}^{\bd-2} \hbar^{-k}D_{\bd -k}b(\theta na)\right]e_{i\hbar^{-1}\lambda, \theta } (\theta na)
	\end{multline*}
	where $D_{\bd-k}\in U(\fN_{\IC})U(\fA_{\IC})$ depends only on $\lambda$ and $H$.
	
	A term in $\hbar^{-k}$ can only arise if $e_{i\hbar^{-1}\lambda, \theta }$ is differentiated $k$ times;
	but $Z$ being of degree $\bar d$, we see then that $D_{\bar d-k}$ can be of order $\bar d-k$ at most.
	The last term, when multiplied by $\hbar^d$, becomes $\sum_{k=2}^{\bar d}\hbar^k D_k b$.  
	We do not know a priori if the function $D_{\bd -k}b$ (defined on $G$) is $M$-invariant, but the sum $\sum_{k=0}^{\bd-2} \hbar^{-k}D_{\bd -k}b$ necessarily defines an $M$-invariant function on $G$, since all the other terms do. Since $\hbar$ is arbitrary, we see that each $D_k$ must necessarily send an $M$-invariant function to an $M$-invariant function. 
	
	Finally, we write
	\begin{eqnarray*}\tau_{-\rho}H(-i\bullet).[b_\theta.e_{i\hbar^{-1}\lambda, \theta }](\theta naM)&=&
	H(-i\bullet)[b_\theta.e_{i\hbar^{-1}\lambda -\rho, \theta }]. e_{0, \theta }(\theta naM)\\
	&=&[\tau_{i\hbar^{-1}\lambda}H(-i\bullet).b_\theta]. e_{i\hbar^{-1}\lambda, \theta }(\theta naM).
	\end{eqnarray*}
	When multiplying by $\hbar^{\bar d}$, and using the Taylor expansion of $H$ at $\lambda$, we have
	$$\hbar^{\bar d} \tau_{i\hbar^{-1}\lambda}H(-i\bullet) 
	=H(\lambda)-i\hbar dH(\lambda) +\sum_{k=2}^{\bar d}
	 \frac{(-i\hbar)^k}{k!} d^{(k)}H(\lambda).$$

	\end{proof}
	
We will refer to a function of the form $x\mapsto	b_\theta(x)e_{i\hbar^{-1}\lambda, \theta }(x)$
as a {\em WKB state}, using the language of semiclassical analysis.

	%%%%%%%%%%%%%%%%%%%%%%%%%%%%%%%%%%%%%%%%%%%%%%%%%%
	\subsection{Symplectic lift}
	%%%%%%%%%%%%%%%%%%%%%%%%%%%%%%%%%%%%%%%%%%%%%%%%%%
	Let $\psi$ be a $\cD$-eigenfunction, of spectral parameter $\nu$.
	We let $\hbar=\norm{\nu}^{-1}$ (the choice of the norm here is arbitrary, one can take the Killing norm for instance).
	We sometimes write $\psi=\psi_\nu$ to indicate the spectral parameter,
	but this notation is imprecise in that $\psi$ may not be uniquely
	determined by $\nu$.

	To $\psi_\nu$ we attach a distribution $\tilde \mu_\psi$
	(sometimes denoted $\tilde \mu_\nu$) on $T^*\bY$:
	for $a\in C_c^\infty(T^*\bY)$ set
	$$\tilde \mu_\psi(a)=\left\la \psi, \Op_\hbar(a)\psi\right\ra_{L^2(\bY)}$$

	As described in Section \ref{s:intro} we are trying to classify weak-* limits of the
	distibutions $\tilde \mu_\nu$ in the limit $\nu\to\infty$.  We fix
	such a limit (``semiclassical measure'') $ \mu$ and a sequence
	$(\psi_j)_{j\in\IN} = (\psi_{\nu_j})_{j\in\IN}$ of eigenfunctions
	such that the corresponding sequence $(\tilde\mu_{\nu_j})$ converges
	weak-* to $ \mu$.  In the sequel we write $\nu$ for $\nu_j$. We assume that $\nu$ goes to infinity in the conditions of paragraph \ref{s:semi}, the limit $\nu_\infty$ assumed to be regular.
	We let $\hbar=\norm{\nu}^{-1}.$ Writing $\Lambda=\Lambda_\nu=\hbar \Im m(\nu)$ we have
	$\Lambda\To\Lambda_\infty=\Im m(\nu_\infty)=-i\nu_\infty$.
	Note that $\Re e(\nu)$ is bounded
        \cite[\S16.5(7) \& Thm.\ 16.6]{Knapp86}), so that
	$\hbar\nu=i\lambda_\nu +O(\hbar)$.  Necessarily $\nu_\infty$ is purely imaginary.
	
		With the notations of Section \ref{s:HC}, the state $\psi_\nu$ satisfies
	\begin{equation}\label{e:eigenstate}
	  \Op_\hbar(H).\psi_\nu= H(-i\hbar\nu)\psi_\nu\,\, 
	\end{equation}
	for all $H\in\cH$.  From now on, we fix a Hamiltonian $H\in \cH$. The letter $H$ will stand for two different objects that are canonically related: a function $H$ on $T^*(G/K)$ ($G$-invariant and polynomial in the fibers of the projection $T^*(G/K)\To G/K$), a $W$-invariant
	 polynomial function on $\fA^*$, an element of $U(\fA)^W$.

	We denote $X_\Lambda=dH(\Lambda)\in \fA$. Since $\Lambda$ is only defined up to an element of $W$, so is $X_\Lambda$. One can assume that $\alpha(X_{\Lambda_\infty})\geq 0$ for all $\alpha\in\Delta^+$.
	For simplicity (and without loss of generality), we will also assume that $ \Lambda_\infty$ belongs to 
	the Weyl chamber $C_\Pi$.

	%For convenience, we choose $\Lambda$ such that $X_\Lambda$ is in the positive Weyl chamber (or one of its walls).

	\vspace{.5cm}
	{\bf Other miscellaneous notations:} $d$ is the dimension of $G/K$, $r$ is the rank, and $J$ the dimension of $N$ (so that $d=r+J$). We call ${\tilde J}$ the number of roots.  
	We index the positive roots $\alpha_1, \ldots, \alpha_{\tilde J}$ in such a way that $\alpha_1(X_{\Lambda_\infty})\leq \alpha_2(X_{\Lambda_\infty})\leq\ldots \leq\alpha_{\tilde J}(X_{\Lambda_\infty})$ (with our previous notations, we have $\alpha_{\tilde J}(X_{\Lambda_\infty})=\chi_{\max}(H)$). We fix $\cK$ as in Theorem \ref{t:main}, and we denote
	$j_0=j_0(X_{\Lambda_\infty})$ the largest index $j$ such that
	$\alpha_j(X_{\Lambda_\infty})< \frac1{2\cK}$.
	
With $\wl\in W$ the long element, we set:
$\lnf  = \oplus_{j>j_0}     \lieg_{\alpha_j}$,
$\lns  = \oplus_{j\leq j_0} \lieg_{\alpha_j}$,
$\lnbf = \oplus_{j>j_0}     \lieg_{\wl.\alpha_j}$,
$\lnbs = \oplus_{j\leq j_0} \lieg_{\wl.\alpha_j}$
$J_0=\dim\lns=\sum_{j\leq j_0} m_{\alpha_j}$.
The spaces $\lnf$ and $\lnbf$ are subalgebras, in fact ideals,
in $\lien$, $\lienb$ respectively;
they generate subgroups $\Nf, \Nbf$ that are normal in $N, \overline{N}$
respectively.
	%In $G$, every element
	%$g$
	%close to the identity can be written in a unique way as $g=\exp(\sum_{\alpha}Y_\alpha)$, with
	%$Y_\alpha\in\fG_\alpha$ close to zero (here we denote $\fG_0=\fA$). We will denote $Y_\alpha=\log_\alpha(n)$. We will denote $V\subset G$ a neighbourhood of identity, symmetric, and small enough so that
	%$$2\norm{\log_\alpha(g_1)+\log_\alpha(g_2)}\geq \norm{\log_{\alpha}(g_1g_2)}\geq \frac12 \norm{\log_\alpha(g_1)+\log_\alpha(g_2)}$$
	%for $g_1, g_2\in V$.

	%necessaire ? correct ?

	%We consider the flow $e^{tX_\Lambda}$ acting on $G/M$ on the right. It turns out that for our analysis the good directions are the {\em expanded} ones: if we want $N$ to correspond to the expanding direction we will have to reverse time. We consider $e^{-tX_\Lambda}$ with $t\geq 0$.
	%In the following argument, `unstable' systematically to refers to $N$ and `stable' refers to $\overline{N}$. The vector $X_\Lambda$ may be singular, so that these stable or unstable spaces can also contain neutral directions for the flow.
	% a repenser

	%%%%%%%%%%%%%%%%%%%%%%%%%%%%%%%%%%%%%%%%%%%%%%%%
	\section{The WKB Ansatz \label{s:WKB}}
	%%%%%%%%%%%%%%%%%%%%%%%%%%%%%%%%%%%%%%%%%%%%%%

We now start the proof of Theorem \ref{t:main}. We first describe how the operator $P_{\bom}^\chi$
acts on WKB states. In Section \ref{s:CS}, we will use the fact that these states form a kind of basis
to estimate the norm of the operator.
	%%%%%%%%%%%%%%%%%%%%%%%%%%%%%%%%%%%%%%%%%%%%%%%%%% 
	\subsection{Goal of this section}
	%%%%%%%%%%%%%%%%%%%%%%%%%%%%%%%%%%%%%%%%%%%%%%%%%%
	Fix a sequence ${\bom}=(\om_{-T},\cdots, \om_{-1},\om_0, \cdots \om_{T-1})$, of length $2T$ chosen so that $T\eta \leq \cK|\log \hbar|$.
Theorem \ref{t:main} requires us to estimate the norm of the operator $P_{\bom}^\chi$
acting on $L^2(\bY)$ (for a suitable choice of the time step $\eta$). This operator is the same as
	$U^{-(T-1)\eta}\cP$
	where
	$$\cP=P_{\omega_{T-1}}U^\eta\ldots U^\eta P_{\omega_{0}}^{1/2}\Op_\hbar(\chi)P_{\omega_{0}}^{1/2}U^\eta\ldots P_{\omega_{-T+1}}U^\eta P_{\omega_{-T}} ,$$
	where we recall that $$U^t=\exp(i\hbar^{-1} t\Op_\hbar(H)).$$
	On the ``energy layer'' $\cE_\lambda$, $U^t$ quantizes the action of $e^{-tX_\lambda}$, in other words the time $-t$ of the Hamiltonian flow generated by $H$. Under the action of $e^{-tX_\lambda}$ for $t\geq 0$, elements of $\fN$ are expanded and elements of $\bar\fN$ are contracted (the vector $X_\Lambda$ may be singular, so that these stable or unstable spaces can also contain neutral directions).

	In what follows we estimate the norm of $\cP$.
	To do so, we will first describe how $\cP$ acts on our Fourier basis $e_{i\hbar^{-1}\lambda, \theta}$,
	using the technique of WKB expansion (\S \ref{s:ansatz}). Then, we will use the Cotlar-Stein lemma (\S \ref{s:CS})
	to estimate as precisely as possible the norm of $\cP$.

	The sequence $\omega_{-T},\ldots, \omega_{T-1}$ is fixed throughout this section. Instead of working with functions on $\bY$ we work with functions on $G/K$ that are $\Gamma$-invariant. For instance, $P_\omega$ is the multiplication operator by the $\Gamma$--invariant function $P_\omega$. We assume that each connected component of the support of $P_\omega$ has very small diameter (say $\eps$). We will fix $Q_\omega$, a function in $C_c^\infty(\bS)$ such that $\Pi_\Gamma Q_\omega=P_\omega$ and such that the support of $Q_\omega$ has diameter $\eps$. We also denote $ Q_\omega$ the corresponding multiplication operator. Finally we need to introduce $Q'_\omega$ in $C_c^\infty(\bS)$
	which is identically $1$ on the support of $Q_\omega$ and supported in a set of diameter $2\eps$.
	  
	 We decompose
	\begin{equation}\label{e:cP}\cP=\cS^*\cU_\chi \end{equation}
	where 
	$$\cU_\chi=\Op(\chi)
	P_{\omega_0}^{1/2}U^{\eta}P_{\omega_{-1}}\ldots U^\eta\widehat  P_{\omega_{-T+1}}U^\eta P_{\omega_{-T}}$$
	and
	$$\cS=
	P_{\omega_0}^{1/2}\ldots U^{-\eta}P_{\omega_{T-2}}U^{-\eta}P_{\omega_{T-1}}.$$

	 %%%%%%%%%%%%%%%%%%%%%%%%%%%%%%%%%%%%%%%%%%%%%%%%%
	%\subsection{Stable/unstable matrix bases}
	%%%%%%%%%%%%%%%%%%%%%%%%%%%%%%%%%%%%%%%%%%%%%%%%%%%
	%For $\lambda\in \bar\fA^*_{+}$, we will denote $e^u_{i\hbar^{-1}\lambda, \theta} =e_{i\hbar^{-1}\lambda, \theta}$ and $e^s_{i\hbar^{-1}\lambda, \theta} = \bar e_{-\hbar^{-1}\lambda, \theta}= e_{-i\hbar^{-1}\lambda, \theta}$. The $s$ means stable and the $u$ means unstable.  
	%It is not clear that it is a good name. Moreover, in view of the sequel, it would make more sense to parametrize the stable basis by $-\Ad(w).\lambda$.

	%%%%%%%%%%%%%%%%%%%%%%%%%%%%%%%%%%%%%%%%%%%%%%%%%%%
	\subsection{The WKB Ansatz for the Schr\"odinger propagator \label{s:ansatz}}
	%%%%%%%%%%%%%%%%%%%%%%%%%%%%%%%%%%%%%%%%%%%%%%%%%%%
	We recall some standard calculations, already done in~\cite{AN07}, with some additional simplifications
	coming from the fact that the functions $e_{i\hbar^{-1}\lambda, \theta}$ are eigenfunctions of $\Op_\hbar(H)$.

	On $\bS$, let us try to solve
	$$-i\hbar \frac{\partial \tilde u}{\partial t}=\Op_\hbar(H)\tilde u,$$
	in other words
	$$\tilde u(t)=U^t \tilde u(0),$$
	with initial condition the WKB state
	$\tilde u(0, x)=a_\hbar(0, x)e_{i\hbar^{-1}\lambda, \theta}(x).$ We only consider $t\geq 0$.
	We assume that $a_\hbar$ is compactly supported and has an asymptotic expansion in all $C^l$ norms as $a_\hbar\sim \sum_{k\geq 0}
	\hbar^k a_k$. We look for approximate solution up to order $\hbar^M$, in the
	form
	$$u(t, x)=e^{\frac{itH(\lambda)}\hbar}e_{i\hbar^{-1}\lambda, \theta}(x)a_\hbar(t, x)=e^{\frac{itH(\lambda)}\hbar}e_{i\hbar^{-1}\lambda, \theta}(x)\sum_{k=0}^{M-1}
	\hbar^k a_k(t, x).$$Let us denote
	\begin{equation}\label{e:firstansatz}u(t, x)=e^{\frac{itH(\lambda)}\hbar}e_{i\hbar^{-1}\lambda, \theta}(x)a_\hbar(t, x, \theta, \lambda)=e^{\frac{itH(\lambda)}\hbar}e_{i\hbar^{-1}\lambda, \theta}(x)\sum_{k=0}^{M-1}
	\hbar^k a_k(t, x,\theta, \lambda)
	\end{equation}
	to keep track of the dependence on $\theta$ and $\lambda$; the pair $(x, \theta)$ then represents an element of $G/K\times
	G/P_0=G/M$.
	Identifying powers of $\hbar$, and using Lemma \ref{l:crucial}, we find the conditions:
	\begin{equation} \label{e:mainBKW}
	\begin{cases} \frac{\partial a_0}{\partial t}(x, \theta)= [dH(\lambda). a_0](x, \theta)\quad
	\mbox{($0$-th transport equation)}\,\,
	\\
	\\
	\frac{\partial a_k}{\partial t}(x, \theta)= [dH(\lambda). a_k](x, \theta)+i\sum_{l=2}^{\bar d}\sum_{l+m=k+1}D_l a_m (x, \theta)
	\quad\mbox{($k$-th transport equation)}\,\,.
	\end{cases}
	\end{equation}
	%\begin{rem}If we compare with the analogous set of equations (3.6) of~\cite{AN07}, we see that in the present situation:\\
	%$\bullet$ the Hamilton-Jacobi equation $\frac{\partial U}{\partial t}-H(x, d_x U)=0$ is automatically solved by the function $U:x\mapsto \lambda.B(x, \theta)+tH(\lambda)$. This corresponds to the fact that 
	%$e_{i\hbar^{-1}\lambda, \theta}$ is an eigenfunction of $\Op_\hbar(H)$ with eigenvalue $H(\lambda)$.\\
	%$\bullet$ the transport equation satisfied by $a_0$ is simpler, it only involves tranport by $e^{tdH(\lambda)}$ but no renormalization by the (square root) Jacobian of that map. This Jacobian is already taken into account in the term $e^{\rho.B(x, \theta)}$ that enters the definition of $e_{i\hbar^{-1}\lambda, \theta}$.
	%\end{rem}
	The equations \eqref{e:mainBKW} can be solved explicitly by
	$$a_0(t, (x, \theta), \lambda)=a_0(0, (x, \theta)e^{tX_\lambda}, \lambda),$$
	in other words
	$$a_0(t)=R(e^{tX_\lambda})a_0(0),$$
	where $R$ here denotes the action of $A$ on functions on $G/M$ by right translation; and
	$$a_k(t)=R(e^{tX_\lambda})a_k(0)+ \int_0^t R(e^{(t-s)X_\lambda})
	\left(i\sum_{l=2}^{\bar d}\sum_{l+m=k+1}D_l a_m (s,x, \theta)\right)ds.$$
	If we now define $u$ by \eqref{e:firstansatz},
	$u$ solves 
	$$-i\hbar \frac{\partial \tilde u}{\partial t}=\Op_\hbar(H)\tilde u
	- e^{\frac{itH(\lambda)}\hbar}e_{i\hbar^{-1}\lambda, \theta}\left[ \sum_{l=2}^{\bar d}\sum_{k= M+1-l}^{M-1}
	\hbar^{k+l}D_l a_k\right]
	$$
	and thus
	\begin{eqnarray*}\norm{u(t)-U^t u(0)}_{L^2(\bS)}&\leq& \int_0^t
	\left[ \sum_{l=2}^{\bar d}\sum_{k= M+1-l}^{M-1}
	\hbar^{k+l-1}\norm{D_l a_k(s)}_{L^2(\bS)}\right]ds \label{e:remainder}\\
	&\leq&  t  e^{(2M+\bar d-2)t \max_{\alpha\in\Delta^+}\alpha(X_\lambda)^-}\left[ \sum_{l=2}^{\bar d}\sum_{k= M+1-l}^{M-1}
	\hbar^{k+l-1}\sum_{j=0}^k\norm{a_{k-j}(0)}_{C^{2j+l}}\right]\\
	&\leq & C t\hbar^M  e^{(2M+\bar d-2)t \max_{\alpha\in\Delta^+}\alpha(X_\lambda)^-} \left[  \sum_{k=0}^{M-1}\norm{a_{k}(0)}_{C^{2(M-k)+\bar d -2}}\right].
	\end{eqnarray*}
	Since $D_k$ belongs to $U(\fN_{\IC})U(\fA_{\IC})$, in the co-ordinates
        $(x, \theta)$ it only involves differentiation with respect to $x$. We also recall that $D_k$ is of order $k$.
	We have used the following estimate on the flow $R(e^{tX_\lambda})$ (for $t\geq 0$)~:
	$$\norm{\frac{d^N}{dx^N}a((x, \theta)e^{tX_\lambda})} \leq
		e^{-tN \min_{\alpha\in\Delta^+}\alpha(X_\lambda)}
		\norm{\frac{d^N}{dx^N}a((x, \theta)}$$
	and we have denoted $x^-=\max(-x, 0).$  

	\begin{rem}In what follows we will always have $\lambda\in\supp(\chi)$, where by assumption $\chi$ is supported on a tubular neighbourhood of size $\epsilon $ of $\cE_{\Lambda_\infty}$, and $\alpha(\Lambda_\infty)\geq 0$ for $\alpha\in \Delta^+$. For such $\lambda$  we have $\alpha(X_\lambda)\geq -\eps$ for all $\alpha\in\Delta^+$. 
	%\spade j'ai pas l'impression de l'avoir dit
	We see that our approximation method makes sense if $t$ is restricted by $\hbar^M  e^{(2M+\bar d-2)t\epsilon} \ll 1$. Since $\eps$ can be chosen arbitrarily small, we can assume that the WKB approximation is good for $t\leq 3\cK|\log\hbar|$.
	\end{rem}

	\begin{rem} On the quotient $\bY=\Gamma\backslash \bS$, the same method applies to find an approximate
	solution of $U^t \Pi_\Gamma u(0)$ in the form $\Pi_\Gamma u(t)$, with the same bound
	\begin{equation}\norm{\Pi_\Gamma u(t)-U^t \Pi_\Gamma u (0)}_{L^2(\bY)} \leq C t\hbar^M e^{\eps t(2M+\bar d-2)}  \left[  \sum_{k=0}^{M-1}\norm{a_{k}(0)}_{C^{2(M-k)+\bar d -2}}\right],
	\end{equation}
	provided that the projection $\bS\To \bY$ is bijective when restricted to the support of $a_\hbar(t)$. If $\lambda$ stays in a compact set and if the support of $a_\hbar(0)$ has small enough diameter $\eps$, this condition will be satisfied
	in a time interval $t\in [0, T_0]$. In the applications below, we may and will always assume that $\eta <T_0$.
	\end{rem}

	%changer N en M 
	 
	We can iterate the previous WKB construction $T$ times to get the following description of the action of $\cU_\chi$ on $\Pi_\Gamma Q'_{\omega_{-T}} e_{i\hbar^{-1}\lambda, \theta}$ (the induction argument to control the remainders at each step is the same as in~\cite{AN07} and we won't repeat it here):
	\begin{prop}\label{p:ansatzforU1} 
	\begin{equation}\label{e:prop6}\cU_\chi(\Pi_\Gamma Q'_{\omega_{-T}}e_{i\hbar^{-1}\lambda, \theta})=
	\Pi_\Gamma\left[e^{\frac{iT\eta H(\lambda)}\hbar}e_{i\hbar^{-1}\lambda, \theta}A_M^{(T)} (\bullet, \theta, \lambda)
	\right]+\cO_{L^2(\bY)}(\hbar^M)\norm{ Q^{'}_{\omega_{-T}}e_{i\hbar^{-1}\lambda, \theta}}_{L^2 (\bS)}
	\end{equation}
	where
$$A_M^{(T)} (x, \theta, \lambda)=\sum_{k=0}^{M-1}\hbar^k a_k^{(T)}(x, \theta, \lambda).$$
The function  $a_0^{(T)}(x, \theta, \lambda)$ is equal to
$$a_0^{(T)}(x, \theta, \lambda)=\chi(\lambda)\,P^{1/2}_{\omega_0}(x)P_{\omega_{-1}}((x, \theta)e^{\eta X_\lambda})
P_{\omega_{-2}}((x, \theta)e^{2\eta X_\lambda})\ldots Q_{\omega_{-T}}((x, \theta)e^{T\eta X_\lambda}),$$
where we have lifted the functions $P_\omega$ (originally defined on $G/K$) to $G/M=G/K\times G/P_0$. The functions $a_k^{(T)}$ have the same support as $a_0^{(T)}$. Moreover, 
if we consider $a_k^{(T)}$ as a function of $(x, \theta)$, that is, as a function on $G/M$, we have the following bound
 $$\norm{Z_\alpha^m a_k^{(T)}}\leq P_{k, m, Z_\alpha}(T)\sup_{j=0,\ldots T}\{e^{-(m+2k)j\eta\,\, \alpha(X_\lambda)}\}$$
 if $Z_\alpha$ belongs to $\fG_\alpha$ ($P_{k, m, Z_\alpha}(T)$ is polynomial in $T$). In particular, for $\alpha\in\Delta^+$,
 $$\norm{Z_\alpha^m a_k^{(T)}}\leq P_{k, m, Z_\alpha}(T) e^{(m+2k)T\eta\,\,  \eps}$$

\end{prop}
The energy parameter $\lambda$ will always stay $\eps$-close to $\Lambda_\infty$.
Recall that we denote by the same letter $\eps$ the diameter of the support of each $Q_\omega$.
We choose $\eps$ and $\eta$ (the time step) small enough to ensure the following: there exists $\gamma=\gamma_{\omega_{-T},\ldots, \omega_0}\in \Gamma$ (independent of $\theta$ or $\lambda$) such that 
\begin{equation}\label{e:defgamma}a_0^{(T)}(x, \theta, \lambda)=\chi(\lambda)\,Q^{1/2}_{\omega_0}\!\!\circ \!\gamma\,\,^{-1}(x)P_{\omega_{-1}}((x, \theta)e^{\eta X_\lambda})
P_{\omega_{-2}}((x, \theta)e^{2\eta X_\lambda})\ldots Q_{\omega_{-T}}((x, \theta)e^{T\eta X_\lambda}).
\end{equation}
This means that the function $a_0^{(T)}(\bullet, \theta, \lambda)$ is supported in a single connected component of the support of $P^{1/2}_{\omega_0}$.

We will also use the following variant: 
\begin{prop}\label{p:ansatzforU2} Let $\gamma=\gamma_{\omega_{-T},\ldots, \omega_0}$.
$$ \cU_\chi ( Q^{'}_{\omega_{-T}} \!\!\circ \!\gamma\,\, \,\,\,\,e_{i\hbar^{-1}\lambda, \theta})
= 
e^{\frac{iT\eta H(\lambda)}\hbar}e_{i\hbar^{-1}\lambda, \theta}(x)A_M^{(T)} \!\!\circ \!\gamma\,\,(x, \theta, \lambda)
+\cO(\hbar^M)\norm{ Q^{'}_{\omega_{-T}} \!\!\circ \!\gamma\,\, e_{i\hbar^{-1}\lambda, \theta}}
 $$where
$$A_M^{(T)} (x, \theta, \lambda)=\sum_{k=0}^{M-1}\hbar^k a_k^{(T)}(x, \theta, \lambda).$$
%The function  $a_0^{(T)}(x, \theta, \lambda)$ is equal to
%\begin{eqnarray}a_0^{(T)}(x, \theta, \lambda)&=&P_{\omega_0}(x)P_{\omega_{-1}}((x, \theta)e^{\eta X_\lambda})
%P_{\omega_{-2}}((x, \theta)e^{2\eta X_\lambda})\ldots Q_{\omega_{-T}}((x, \theta)e^{T\eta X_\lambda})\\
%&=& 
%Q_{\omega_0}\!\!\circ \!\gamma\,\,^{-1}(x)P_{\omega_{-1}}((x, \theta)e^{\eta X_\lambda})
%P_{\omega_{-2}}((x, \theta)e^{2\eta X_\lambda})\ldots Q_{\omega_{-T}}\!\!\circ \!\gamma\,\, ((x, \theta)e^{T\eta X_\lambda}).
%\end{eqnarray}
%The functions $a_k^{(T)}$ have the same support as $a_0^{(T)}$,  and satisfy the same estimates as in Proposition \ref{p:ansatzforU1}. 
  \end{prop} 
  
%\begin{rem}
%The function $F_ \lambda$ defined in Remark \ref{r:quasiinvariance} has the property that $|F_\lambda(ge^{-tX_\lambda}M)|=|F_\lambda(gM)|e^{-t\rho(X_\lambda)}.$ This implies that
%We shall not use this fact, but it is also possible to see that
%$$\sup_x |Q_{\omega_{-T}}((x, \theta)e^{T\eta X_\lambda})e_{i\hbar^{-1}\lambda, \theta}(x)|=\sup_x |Q_{\omega_{-T}}(x, \theta)e_{i\hbar^{-1}\lambda, \theta}(x)|e^{-t\rho(X_\lambda)}.$$
%Comparing with \eqref{e:prop6} and \eqref{e:defgamma}, and noting that the functions
%$Q^{1/2}_{\omega_0}\!\!\circ \!\gamma\,\,^{-1}(x)  $ are supported on sets of fixed volume, we see that
 %$$\norm{\cU_\chi(\Pi_\Gamma Q'_{\omega_{-T}}e_{i\hbar^{-1}\lambda, \theta})}_{L^2(\bY)}\leq
 %C e^{-T\eta \rho( X_\lambda)}\norm{\Pi_\Gamma Q'_{\omega_{-T}}e_{i\hbar^{-1}\lambda, \theta}}_{L^2(\bY)}.$$
 %In the same way,
 %$$\norm{\cU_\chi ( Q^{'}_{\omega_{-T}} \!\!\circ \!\gamma\,\, e_{i\hbar^{-1}\lambda, \theta})}_{L^2(\bS)}\leq C
 %e^{-T\eta\rho( X_\lambda)} \norm{ Q^{'}_{\omega_{-T}} \!\!\circ \!\gamma\,\, e_{i\hbar^{-1}\lambda, \theta}}_{L^2(\bS)}.$$
 %\end{rem}

\begin{rem}\label{rem:stable}
For the operator $\cS$, analogous results can be obtained if we replace everywhere $\lambda$ by $\wl.\lambda$, $-t$ by $+t$, and the label
$\omega_{-j}$ by $\omega_{+j}$.
\end{rem}

\begin{rem}Let $u, v\in L^2(\bY)$. We explain how the previous Ansatz can be used to estimate the scalar
product $\left\la v, \cU_\chi u\right\ra_{L^2(\bY)}$ (up to a small error).
This is done by decomposing $u$ and $v$, locally, into a combination of the functions $e_{i\hbar^{-1}\lambda, \theta}$ (using the Helgason-Fourier transform), and inputting our Ansatz into this decomposition.

 In more detail, we note that $P_{\omega_{-T}}=  P_{\omega_{-T}}\Pi_\Gamma  Q^{'2}_{\omega_{-T}}$, so that
$\cU_\chi u=
 \cU_\chi \Pi_\Gamma  Q^{'2}_{\omega_{-T}} u$. We use the Fourier decomposition to write
$$Q^{'2}_{\omega_{-T}}u(x)=Q'_{\omega_{-T}}(x)
\int_{\theta\in G/P_0, \lambda\in   \overline{ C_\Pi}}\widehat{Q'_{\omega_{-T}}u}(\lambda, \theta)
e_{i\hbar^{-1}\lambda, \theta}(x)  d\theta |c_\hbar(\lambda)|^{-2}d\lambda.
$$
By Cauchy-Schwarz and the asymptotics of the $c$-function (Remark \ref{r:GK}), we note that
$$\int_{\chi(\lambda)\not=0}|\widehat{Q'_{\omega_{-T}}u}(\lambda, \theta)|d\theta |c_\hbar(\lambda)|^{-2}d\lambda =\cO(\hbar^{-d/2})\norm{u}_{L^2(\bY)},$$
and write
\begin{multline}\label{e:oneofthethings}\left\la v, \cU_\chi u\right\ra_{L^2(\bY)}= 
\left\la v, \cU_\chi \Pi_\Gamma  Q^{'2}_{\omega_{-T}} u\right\ra_{L^2(\bY)} \\
=
\int_{  \chi(\lambda)\not=0 }\widehat{Q'_{\omega_{-T}}u}(\lambda, \theta)
\left\la v, \cU_\chi \Pi_\Gamma Q'_{\omega_{-T}}e_{i\hbar^{-1}\lambda, \theta}\right\ra  d\theta |c_\hbar(\lambda)|^{-2}d\lambda  +\cO(\hbar^\infty)\norm{u}_{L^2(\bY)}\norm{v}_{L^2(\bY)}.
\end{multline}
We now use Proposition \ref{p:ansatzforU1} to replace $\cU_\chi$ by the Ansatz,
\begin{eqnarray*}
\left\la v, \cU_\chi \Pi_\Gamma Q'_{\omega_{-T}} e_{i\hbar^{-1}\lambda, \theta}\right\ra_{L^2(\bY)}&=&
\left\la v, e^{\frac{iT\eta H(\lambda)}\hbar}e_{i\hbar^{-1}\lambda, \theta}\,\,\,\,A_M^{(T)} (\bullet, \theta, \lambda)\right\ra_{L^2(\bS)}+\cO(\hbar^M)\norm{v}_{L^2(\bY)}\\
&=& \left\la Q'_{\omega_0}\!\!\circ \!\gamma\,\,^{-1} \,\,.\,\,v,\, e^{\frac{iT\eta H(\lambda)}\hbar}e_{i\hbar^{-1}\lambda, \theta}\,\,\,\,A_M^{(T)} (\bullet, \theta, \lambda)\right\ra_{L^2(\bS)}+\cO(\hbar^M)\norm{v}_{L^2(\bY)}\\
&=& \left\la Q'_{\omega_0} v,\, e^{\frac{iT\eta H(\lambda)}\hbar}e_{i\hbar^{-1}\lambda, \theta}\circ \!\gamma\,\,\,\,\,\, A_M^{(T)} (\gamma \bullet, \theta, \lambda)\right\ra_{L^2(\bS)}+\cO(\hbar^M)\norm{v}_{L^2(\bY)}
%\\
%&=& \left\la  Q'_{\omega_0} v,\,\, \cU_\chi [ Q'_{\omega_{-T}} \!\!\circ \gamma\,\, \,\,\,\,e_{i\hbar^{-1}\lambda, \theta}\circ \gamma\,\,]\right\ra_{L^2(\bS)} +\cO(\hbar^M)\norm{v}_{L^2(\bY)}
\end{eqnarray*}
for $\gamma=\gamma_{\omega_{-T},\ldots, \omega_0}$
defined above.
Thus,
\begin{multline}
\left\la v, \cU_\chi u\right\ra_{L^2(\bY)}
=\int_{   \chi(\lambda)\not=0}\widehat{Q'_{\omega_{-T}}u}(\lambda, \theta)
 \left\la  Q'_{\omega_0} v,e^{\frac{iT\eta H(\lambda)}\hbar}e_{i\hbar^{-1}\lambda, \theta}\circ \!\gamma\,\,\,\,\,\, A_M^{(T)} (\gamma \bullet, \theta, \lambda)\right\ra_{L^2(\bS)}   d\theta |c_\hbar(\lambda)|^{-2}d\lambda 
\\+\cO(\hbar^{M-d/2})\norm{v}_{L^2(\bY)}\norm{u}_{L^2(\bY)}.
%\\
%=\left\la  Q'_{\omega_0} v, \cU_\chi [  Q^{'2}_{\omega_{-T}} \!\!\circ \!\gamma\,\,\, u]\right\ra_{L^2(\bS)} +\cO(\hbar^{M-d/2})\norm{v}_{L^2(\bY)}\norm{u}_{L^2(\bY)},
\end{multline}
In this last line we see that replacing the exact expression of $\cU_\chi$ by the Ansatz induces an error of $\cO(\hbar^{M-d/2})\norm{v}_{L^2(\bY)}\norm{u}_{L^2(\bY)}$. We will take $M$ very large,
depending on the constant $\cK$ in Theorem \ref{t:main}, so that the error $\cO(\hbar^{M-d/2})$
is negligible compared to the bound announced in the theorem.
\end{rem}

%%%%%%%%%%%%%%%%%%%%%%%%%%%%%%%%%%%%%%%%%%%%%%%%%%%
\section{The Cotlar--Stein argument. \label{s:CS}}
%%%%%%%%%%%%%%%%%%%%%%%%%%%%%%%%%%%%%%%%%%%%%%%%%%

We now use the previous approximations of $\cU_\chi$ and $\cS$ to estimate the norm
of $\cP$. This is done in a much finer, and more technical manner, than in \cite{An, AN07}, because we want to eliminate the slowly expanding/contracting directions.

\subsection{The Cotlar-Stein lemma}
\begin{lem}Let $E, F$ be two Hilbert spaces. Let $(A_\alpha)\in \cL(E, F)$ be a countable family
of bounded linear operators from $E$ to $F$. Assume that for some $R>0$ we have
$$\sup_{\alpha}\sum_{\beta}\norm{A^*_\alpha A_\beta}^{\frac12}\leq R$$ and
$$\sup_{\alpha}\sum_{\beta}\norm{A_\alpha A^*_\beta}^{\frac12}\leq R$$
Then $A=\sum_{\alpha}A_\alpha$ converges strongly and $A$ is a bounded operator with $\norm{A}\leq R$.
\end{lem}
We refer for instance to~\cite{DS99} for the proof.

\subsection{A non-stationary phase lemma}
The following lemma is just a version of integration by parts.
\begin{lem} \label{l:NSP}Let $\Omega$ be an open set in a smooth manifold. Let $Z$ be a vector field on $\Omega$ and $\mu$ be a measure
on $\Omega$ with the property that $\int (Zf) d\mu=\int fJd\mu$ for every smooth function $f$ and for some smooth $J$.

Let $S\in C^\infty(\Omega, \IR)$ and
$a\in C_c^\infty(\Omega)$. Assume that $ZS$ does not vanish. Consider the integral
\begin{equation}\label{e:NSP}I_\hbar=\int e^{\frac{iS(x)}\hbar}a(x)d\mu(x).\end{equation}

Then we have
$I_\hbar= i\hbar \int  e^{\frac{iS(x)}\hbar} D_Za (x)d\mu(x)$,
where the operator $D_Z$ is defined
by
$$D_Z a=Z\left(\frac{a}{ZS}\right)-\frac{aJ}{ZS}.$$
If we iterate this formula $n$ times we get
$$I_\hbar =(i\hbar)^n \int  e^{\frac{iS(x)}\hbar} D^n_Za (x)d\mu(x)$$ and
$D^n_Z$ has the form
$$D^n_Z a=\sum_{m\geq n, k+m\leq 2n, \sum l_j\leq n }f_{k,(l_j),m}\frac{Z^k a Z^{l_1} S\ldots
Z^{l_r} S}{(ZS)^m} $$
where the $f_{k,(l_j),m}(x)$ are smooth functions that do not depend on $a$ nor $S$.\end{lem}

%%%%%%%%%%%%%%%%%%%%%%%%%%%%%%%%%%%%%%%%%%%%%%%%%%
\subsection{Study of several phase functions}
%%%%%%%%%%%%%%%%%%%%%%%%%%%%%%%%%%%%%%%%%%%%%%%%%%%

%$$\cN_{fast}=\{n\in N, H_0(\Ad(n).Y)=0 \mbox{ for all } Y\in \bar\fN_{slow}.$$
%In a neighbourhood of identity, this is a submanifold of $N$, with tangent space at $e$
%$$\{Z\in \fN, H_0([Z, Y])=0 \mbox{ for all } Y\in \bar\fN_{slow}=\fN_{fast}.$$
%To prove that the tangent space coincides with $\fN_{fast}$, consider 
% $Z\in \fN$ such that $H_0([Z, Y])=0$ for all $Y\in \bar\fN_{slow}$. Decompose $Z=\sum Z_\alpha$
% with $Z_\alpha\in\fG_{\alpha}$. Take $Y=\theta(Z_\beta)$: then $H_0([Z, Y])=
%- \left\la Z_\beta, Z_\beta\right\ra H_\beta\in\fA\spadesuit$$
%where $H_\beta\in \fA$ is the coroot and satisfies $\lambda(H_\beta)=\la \lambda, \beta\ra$ for every $\lambda\in \fa^*$. We see that the condition $H_0([Z, Y])=0$ for all $Y\in \bar\fN_{slow}$
%implies that $Z_\beta=0$ for all

\subsubsection{ Sum of two Helgason phase functions} 
%Let $\lambda\in\fA^*$.
%In this section, we will introduce the following notation:
%$$\cN_\lambda=\{n\in N, \alpha.H_0(\Ad(n^{-1}).Y)=0\mbox{ for all }Y\in \bar\fN\mbox{ and for all } \alpha\in\Delta\mbox{ s.t. }\la\lambda, \alpha\ra\not=0.\}$$
%Near identity, it is a submanifold of $N$, with tangent space at identity
%$$ \{Z\in  \fN, \alpha.H_0([Z, Y])=0\mbox{ for all }Y\in \bar\fN\mbox{ and for all } \alpha\in\Delta\mbox{ s.t. }\la\lambda, \alpha\ra\not=0.\}$$
%It is not hard to see that this tangent space is also $\sum_{\alpha\in\Delta^+,\la\lambda, \alpha\ra=0 }\fG_\alpha$. In particular, if $\lambda$ is regular (which will be the case in our discussion), the intersection of $\cN_\lambda$ with a neighbourhood of identity is reduced to a point.
%We also denote
%$$\bar\cN_\lambda=\Ad(w). \cN_{-w.\lambda}$$

\begin{prop}\label{p:crit} (i) Let $g_1P_0, g_2P_0\in G/P_0$ be two points on the boundary. Let $\lambda, \mu\in \overline{ C_\Pi}$ be two elements of the closed nonnegative Weyl chamber. Consider the function on $G/K$,
\begin{equation} \label{e:map}gK\mapsto \lambda.H_0(g_1^{-1}gK)+\mu.H_0(g_2^{-1}gK).
\end{equation}
Then, this map has critical points if and only if $\mu=-\Ad(\wl).\lambda$. 

(ii)  Let $\lambda, \mu\in { C_\Pi}$ be two (regular) elements of the positive Weyl chamber.
Let $g_1P_0, g_2P_0\in G/P_0$ be two points on the boundary, and assume that $g_1^{-1}g_2\in P_0\wl P_0$
(we don't assume here that the conclusion of (i) is satisfied).
Write $g_1^{-1}g_2 = b_1 \wl b_2$ with $b_1, b_2\in P_0$.

Then, the set of critical points for variations of the form
$$t\mapsto  \lambda.H_0(e^{tX}g_1^{-1}gK)+\mu.H_0(e^{tX}g_2^{-1}gK),$$
with $X\in \fN$ is precisely $\{gK, g\in g_1 b_1 A \}$.
Moreover, these critical points are non-degenerate.

%  Or, equivalently,
% $\{gK, g\in g_2 b^{-1}_2 A N^\lambda\}$ where $ N^{\lambda}=\exp(\sum_{\alpha\in\Delta^+, \left\la \lambda,\alpha\right\ra=0} \fG_{\alpha}).$
\end{prop} 

\begin{rem} The set of critical points is $\{gK, g\in g_1 P_0, g\wl\in g_2P_0\}$, that is, the flat in $G/K$ determined by the two boundary
points $g_1P_0, g_2P_0$.

\end{rem}

\begin{proof} (i) It is enough to consider the case $g_1 = e$.
By the Bruhat decomposition, we know that there exists
a unique $w\in W$ such that $g_2\in BwB$, that is,
$g_2=b_1 w b_2$ for some $b_1, b_2\in B$. The map \eqref{e:map}
has the same critical points as the map
\begin{equation} \label{e:map2}
gK\mapsto \lambda.H_0(gK)+\mu.H_0(w^{-1}b_1^{-1}gK),
\end{equation}
and those are the image under $gK\mapsto b_1gK$ of the critical points of
\begin{equation} \label{e:map3}
gK\mapsto \lambda.H_0(gK)+\mu.H_0(w^{-1}gK).
\end{equation}
For $X\in\fA$ the derivative at $t=0$ of
\begin{equation}\label{e:map4}
t\mapsto \lambda.H_0(e^{tX}gK)+\mu.H_0(w^{-1}e^{tX}gK)
\end{equation}
is $\lambda(X)+\mu(\Ad(w^{-1})X)$.
Thus, for the map \eqref{e:map3} to have critical points, we must have
$$\lambda(X)+\mu(\Ad(w^{-1})X)=0$$ for every $X\in\fA$.
Letting $X$ vary over the dual basis to a positive basis of $\fA^*$,
we see that $\mu=-\Ad(w).\lambda$ is nonnegative, and this is only possible if
$\mu=-\Ad(\wl).\lambda$
(this does not necessarily mean that $w=\wl$ if $\lambda$ is not regular).

(ii) Here we assume that $\mu$ and $\lambda$ are regular, and that we are
in the ``generic'' case where $g_1^{-1}g_2\in P_0\wl P_0$.
Starting from \eqref{e:map3}, we now consider variations of the form
\begin{equation}\label{e:map5}
t\mapsto \lambda.H_0(e^{tX}gK)+\mu.H_0(\wl^{-1}e^{tX}gK)
\end{equation}
for $X\in\fN$. The term
$\lambda.H_0(e^{tX}gK)$ is constant, and it remains to deal with
$\mu.H_0(\wl^{-1}e^{tX}gK)$.  Write $g=\wl anK$, $n\in N, a\in A$,
and denote $Y=\Ad(\wl).X\in \bar\fN$, $Y'=\Ad(a^{-1}) Y$. We have
$$\mu.H_0(\wl^{-1}e^{tX}gK) = \mu. H_0(e^{tY}anK) =
  \mu(a)+\mu.H_0(e^{tY'}nK) = \mu(a)+\mu.H_0(n^{-1}e^{tY'}nK).$$
Hence
$$\frac{d}{dt}\mu. H_0(e^{tY}anK) = \mu. H_0(\Ad(n^{-1})Y').$$
We see that the set of critical points of \eqref{e:map5} is the set of
those points $gK$, with $g=\wl anK$ such that $n$ satisfies
$\mu. H_0(\Ad(n^{-1})Y')=0$ for all $Y'\in\bar\fN$.
Since $\mu$ is regular, one can check that this implies $n=e$.
This proves the first assertion of (ii).
  
Finally, assume that we are at a critical point, that is, $gK=aK$
in \eqref{e:map5}.  We calculate the second derivative at $t=0$ of
$t\mapsto\mu.H_0(\wl^{-1}e^{tX}aK)$ when $X\in \fN.$
We keep the same notation as above for $Y$ and $Y'$.

Let $U=Y'-\theta(Y')\in\fK$.
By the Baker-Campbell-Hausdorff formula, we have 
\begin{equation}\label{e:order2}e^{tY'}=e^{t\theta(Y')+\frac{t^2}2[Y', \theta(Y')]+\cO(t^3)}e^{tU}=
e^{t\theta(Y')}e^{\frac{t^2}2[Y', \theta(Y')]+\cO(t^3)}e^{tU}.
\end{equation}  
 Remember that $\theta(Y')\in \fN$, and that $H_0$ is left-$N$-invariant. This calculation shows that
 the second derivative of  $t\mapsto\mu.H_0(\wl^{-1}e^{tX}aK)$ is the quadratic form
 $$X\mapsto\mu\left([Y', \theta(Y')]\right),$$
 where $Y'=\Ad(a^{-1})\Ad(\wl).X$.
This is a non-degenerate quadratic form if $\mu$ is regular.
\end{proof}

\vspace{.5cm}

\subsubsection{Variations with respect to $\overline{N}$}
In this section we need the decomposition $\fG=\fN\oplus \fA\oplus \fM\oplus \bar\fN$. We will denote
$\pi_\fN$, $\pi_\fA$, $\pi_{\bar\fN}$ the corresponding projections.
We note that $\pi_\fA=H_0$, since $\bar\fN\subset \fN+\fK.$
\begin{lem}  \label{l:n-}Fix $n\in N$ and $a\in A$.
Then there exist two neighbourhoods $V_1, V_2$ of $0$ in $\bar\fN$,
and a diffeomorphism $\Psi=\Psi_{na} \colon V_1\To V_2$
such that
$$e^{-Y_1} na  e^{Y_2}\in NA, Y_1\in V_1, Y_2\in V_2\iff Y_2=\Psi(Y_1).$$
Moreover, the differential at $0$ of $\Psi$ (denoted $\Psi'_0$)
preserves the subalgebra $\lnbs$.
Finally, if we write $e^{-Y} na  e^{\Psi(Y)}=n(Y)a(Y)$, we have
$$a'_0.Y = \pi_\fA[\Ad(na)\Psi'_0(Y)].$$
\end{lem}
%remplacer NA par NAM ?\spadesuit
\begin{proof}
We apply the implicit function theorem.  For $Y_1=0$, the differential of
$Y_2\mapsto na  e^{Y_2} (na)^{-1}$ at $Y_2=0$ is $Y_2\mapsto\Ad(na).Y_2$.
What we need to check is the equivalence of
$\pi_{\bar\fN}[\Ad(na).Y_2]=0$ and $Y_2=0$, which is the case since
$\Ad(na)$ preserves $\fN\oplus\fA\oplus\fM$.  So the existence of $\Psi$ is proved, in addition the differential $\Psi'_0$ is defined by
$$Y=\pi_{\bar\fN}[\Ad(na).\Psi'_0.Y]$$ for $Y\in\bar\fN$.  Since $\Ad(na)$ preserves the space $\fN\oplus \fA\oplus \fM\oplus \lnbs$ (without preserving the decomposition, of course), $\Psi'_0.Y$ must belong to $\lnbs$ if $Y$ does.

The last formula is simply obtained by differentiating $e^{-Y} na  e^{\Psi(Y)}=n(Y)a(Y)$.
\end{proof}

For the next lemma we need to recall our two decompositions $\lienb=\sum_{k\leq j_0}\fG_{\wl.\alpha_k}\oplus \sum_{j> j_0}\fG_{\wl.\alpha_k}=\lnbs\oplus\lnbf$ and
$\fN=\sum_{k\leq j_0}\fG_{\alpha_k}\oplus \sum_{j> j_0}\fG_{\alpha_k}=\lns\oplus\lnf$.
The space $\lnf$ is an ideal of $\fN$, and we call $\Nf$ the associated (normal) subgroup.
 
\begin{lem} \label{l:n?}
(i) The set 
$$\{n\in N, H_0(\Ad(n)Y)=0\qquad \forall Y\in \lnbs\}$$
is, near identity, a submanifold of $N$, tangent to $\lnf$.

(ii) If $n$ is close enough to identity, if we write $n=e^{\sum_\alpha T_\alpha}$ with $T_\alpha\in \fG_{\alpha}$ ($\alpha\in\Delta^+$),
and if $\mu\in\fA^*$ is a regular element, we have
$$\left|\mu.H_0\left(\Ad(n)\frac{\theta(T_\beta)}{\norm{\theta(T_\beta)}}\right)\right|\geq C_\mu\norm{T_\beta}$$
with $C_\mu >0$.

%(i) Fix $n=e^Z\in N$, 
%where $Z=Z_{slow}+Z_{fast}$. For any $a\in A$, any
%$n'\in N_{fast}$, and any $Y\in \bar\fN_{slow}$,
%$$\pi_{\fA}[\Ad(ann').Y]=\pi_{\fA}(\Ad(e^{Z_{slow}}Y).$$

%(ii) Assume that $\pi_{\fA}[\Ad(ann').Y]=0$ for all $Y\in \bar\fN_{slow}$. Then $n\in N_{fast}$.
\end{lem}

\begin{proof}
The differential of $n\mapsto H_0(\Ad(n)Y)$ is $Z\mapsto H_0([Z, Y])$ ($Z\in\lien$).
Write $Z=Z_{slow}+Z_{fast}$,
$Z_{slow}=\sum Z_\alpha$, with $Z_\alpha\in\fG_\alpha$, and take $Y =\theta(Z_\beta)$ for some $\beta$. We have $H_0([Z, Y])=-\la Z_\beta, Z_\beta\ra H_\beta$ where $H_\beta\in \fA$ is the coroot \cite[Ch.\ VI \S5, Prop.\ 6.52]{Kn}. Note that $\mu(H_\beta)=\la \mu, \beta\ra$, hence $\mu(H_\beta)\not=0$ if $\mu$ is regular.
This proves the lemma.

\end{proof}

%%%%%%%%%%%%%%%%%%%%%%%%%%%%%%%%%%%%%%%%%%%%%%%%%
\subsection{First decomposition of $\cP$}
%%%%%%%%%%%%%%%%%%%%%%%%%%%%%%%%%%%%%%%%%%%%%%%%%%%

We want to use the Cotlar-Stein lemma to estimate the norm of the operator $\cP$, defined in \eqref{e:cP}. To do so, we will decompose $\cP$ into many pieces.
 Our first decomposition of $\cP$ is obtained by covering the boundary $G/P_0$ by a finite number
of small sets $\Omega_1, \ldots, \Omega_M$ described below. We use the fact that there is a neighbourhood $\Omega$ of $eP_0$ in $G/P_0$ that is diffeomorphic to a neighbourhood of $e$ in $\overline{N}$, via the map 
\begin{eqnarray*}
\overline{N}&\To& G/P_0\\
\bar n &\mapsto& \bar nP_0. 
\end{eqnarray*}
Using compactness, we can find an open cover of $G/P_0$ by a finite number
of open sets $\Omega_1, \ldots, \Omega_M$ such that, for every $m$, there exists $g_m\in G$ 
with $\Omega_m\subset g_m\Omega\subset g_m \overline{N} P_0$. Introduce a family of smooth functions $\chi_{\Omega_m}$ on $G/P_0$ such that $\chi_{\Omega_m}$ is supported inside $\Omega_m$ and $\sum_m \chi_{\Omega_m}\equiv 1$. We then define the pseudodifferential operators
$$Q_{m}u(x)=\int \widehat  u(\wl.\lambda,k)Q'_{\omega_0}(x)
\chi_{\Omega_m}(k)  
e_{i\hbar^{-1}\wl.\lambda, k} dk|c_\hbar(\lambda)|^{-2}d\lambda ,$$
and
\begin{equation*}\cP_{m}u= \Pi_\Gamma \cS^* Q^*_{ m}\cU_\chi  u
 \end{equation*}

Obviously, $\cP = \sum_m \cP_{m}$.
The sum over $m$ is finite, and we now fix $ m$. The variable $k$ stays in $g_m \overline{N}P_0$.

\begin{rem} Let $\gamma=\gamma_{\omega_{T-1},\ldots, \omega_0}$ defined as in \eqref{e:defgamma}. Proposition \eqref{p:ansatzforU2} (and Remark \ref{rem:stable}) can be generalized to  
\begin{equation}\label{e:ansatzforS}
Q_{m} \cS \left(Q'_{\omega_{T-1}}\!\!\circ \gamma\,\,\,\, e_{i\hbar^{-1}\wl.\mu, k}\right)
= 
e^{ \frac{-iT\eta H(\mu)}\hbar}e_{i\hbar^{-1}\wl.\mu, k}B_M^{(T)} \!\circ \gamma\,\, ( x, k, \wl.\mu )+
  \cO_{L^2(\bS)}(\hbar^{M})\norm{Q'_{\omega_T}\!\circ \gamma\,\,\,\, e_{i\hbar^{-1}\wl.\mu, k} }_{L^2( \bS)},\end{equation}
  where now
  $$B_M^{(T)} (x, k, \wl.\mu )=\sum_{k=0}^{M-1}\hbar^k b_k^{(T)}(x, k, \wl.\mu),$$
 \begin{multline}\label{e:conditions}b_0^{(T)}(x, k, \wl.\mu)= \chi_{\Omega_m}(k)P^{1/2}_{\omega_0}(x)P_{\omega_{1}}((x, k)e^{-\eta X_{\wl.\mu}})
P_{\omega_{2}}((x, k)e^{-2\eta X_{\wl.\mu}})\ldots Q_{\omega_{T-1}}((x, k)e^{-(T-1)\eta X_{\wl.\mu}}) \\
 = \chi_{\Omega_m}(k)Q^{1/2}_{\omega_0}\circ\gamma^{-1}(x)P_{\omega_{1}}((x, \theta)e^{-\eta X_{\wl.\mu}})
P_{\omega_{2}}((x, k)e^{-2\eta X_{\wl.\mu}})\ldots Q_{\omega_{T-1}}((x, k)e^{-(T-1)\eta X_{\wl.\mu}})
\end{multline}
and the next terms have the same support as the leading one (their derivatives are bounded the same way as in Proposition \ref{p:ansatzforU1}).
 \end{rem}
 
 In the next paragraphs we will concentrate our attention on brackets of the form:
$$\left\la Q'_{\omega_{T-1}}\!\!\circ \!\gamma_2\,\,\,\, e_{i\hbar^{-1}\wl.\mu, k}\, , \,
 \cS^* Q^*_m  \cU_\chi Q'_{\omega_{-T}}\!\!\circ \!\gamma_1\,\, \,\,e_{i\hbar^{-1}\lambda, \theta}\right\ra_{L^2(\bS)},$$
for $\lambda, \mu\in C_\Pi$, $\theta, k\in G/P_0$. We take $\gamma_1=\gamma_{\omega_{-T},\ldots, \omega_0}$ and $\gamma_2=\gamma_{\omega_{T-1},\ldots, \omega_0}$ as defined in \eqref{e:defgamma}. These are none other than the matrix elements of the operator $\cP_m$ in the Fourier basis $e_{i\hbar^{-1}\lambda, \theta}$.

%%%%%%%%%%%%%%%%%%%%%%%%%%%%%%%%%%%%%%%%%%%%%%%%%
\subsection{Second decomposition of $\cP$}
%%%%%%%%%%%%%%%%%%%%%%%%%%%%%%%%%%%%%%%%%%%%%%%%%%%
The index $m$ being fixed, we will apply the Cotlar-Stein lemma to
bound the norm of $\cP_{m}$.  We decompose $\cP_{m}$ as a sum of
countably many operators, and this decomposition is much more involved.

We have assumed that we have a diffeomorphism from a relatively compact
subset of $\overline{N}$ to $\Omega_m$: $\bar n_1\mapsto g_m \bar n_1 P_0$.
We can write the Haar measure on $\Omega_m$ as
$dk=\Jac(\bar n_1) d\bar n_1$, where $\Jac$ is a smooth function on $\overline{N}$
(we suppress from the notation its dependence on $g_m$).
An element $(x, k)\in G/K\times \Omega_m$ corresponding to the point
$g_m \bar n_1 n_1a_1M\in G/M$ can also be represented as
$(g_m \bar n_1 n_1a_1K, g_m \bar n_1)\in G/K \times g_m\overline{N}$.
Accordingly we now write
denote $e_{i\hbar^{-1}\wl.\mu, g_m\bar n_1}$ for $e_{i\hbar^{-1}\wl.\mu, k}.$

Let us look at a scalar product
$\left\la  Q_m \cS Q'_{\omega_{T-1}}\!\!\circ \!\gamma_2\,
  e_{i\hbar^{-1}\wl.\mu\, , g_m \bar n_1 P_0},\,\, \cU_\chi
   Q^{'}_{\omega_{-T}}\circ\gamma_1\,e_{i\hbar^{-1}\lambda, \theta}\right\ra$.
We only need to consider the generic case where
$\theta \in g_m\bar n_1 P_0 \wl P_0$, that is, $\theta$ is of the form
$g_m\bar n_1 n_1 \wl P_0$ (with $n_1\in N$).
In addition, we always assume that $\lambda$ and $\mu$ are regular.
Proposition \ref{p:crit} (ii) tells us that the stationary points of the
phase function
$$gK\mapsto \lambda.H_0(\theta^{-1}gK)-(\wl.\mu).H_0 (k^{-1}gK),\qquad k=g_m\bar n_1 P_0$$
with respect to variations
$$ (g_m\bar n_1 n_1)e^{tX} (g_m\bar n_1 n_1)^{-1}gK   ,\qquad X\in\fN,$$
are the points of the form $gK=g_m\bar n_1 n_1  a_1K$ with $a_1\in A$.
Thus the set of critical points is of codimension $J$.
The stationary phase method then gives:
\begin{multline}\label{e:development}
 \left\la  Q_m \cS Q'_{\omega_{T-1}}\!\!\circ \!\gamma_2\,
           e_{i\hbar^{-1}\wl.\mu, g_m \bar n_1 P_0},\,\, \cU_\chi
             Q^{'}_{\omega_{-T}}\circ\gamma_1\,e_{i\hbar^{-1}\lambda, \theta}
 \right\ra
\\
 = \hbar^{J/2}\int_{a_1\in A} d(\lambda, a_1)
      C_\hbar\left(g_m\bar n_1 n_1 a_1 M, \lambda, \wl.\mu\right) 
      \bar e_{i\hbar^{-1}\wl.\mu,g_m \bar n_1P_0 } (g_m\bar n_1 n_1  a_1K)
      \\
      e_{i\hbar^{-1}\lambda,g_m \bar n_1 n_1 \wl P_0}(g_m\bar n_1 n_1 a_1K)\,
      d a_1 
\end{multline}
where  $ C_\hbar\left(g_m\bar n_1 n_1 a_1 M, \lambda, \wl.\mu\right)
   \sim\sum \hbar^k c_k\left(g_m\bar n_1 n_1 a_1 M, \lambda, \wl.\mu\right) $
and
\begin{multline*}
c_0 \left( g_m\bar n_1 n_1 a_1 M, \lambda, \wl.\mu \right) = \\
\left( A^{(T)}_M\circ\gamma_1(g_m\bar n_1 n_1 a_1 \wl M, \lambda) \right)
\left( \bar B^{(T)}_M\circ\gamma_2( g_m\bar n_1 n_1   a_1 M , \wl.\mu) \right)
.\end{multline*}
(and the next terms have the same support as the leading one). The term $d(\lambda, a_1)$
is the prefactor involving the hessian of the phase function in the application of the method
of stationary phase, it is a smooth function.
So the asymptotics of our scalar product only takes into account the elements $g_m \bar n_1 n_1 a_1 M$ with
$$A^{(T)}_M\circ\gamma_1(g_m\bar n_1 n_1   a_1\wl M , \lambda)\bar B^{(T)}_M\circ\gamma_2( g_m\bar n_1 n_1   a_1 M ,\wl. \mu)
\not = 0.$$

\begin{lem} Assume that the diameter of $\Omega$ and of $\supp Q_{\omega_0}$ is smaller than  $\eps$. Then
there exist $n_{ 0}\in N$ and $a_0$ in $A$ such that
$$B^{(T)}_M\circ\gamma_2 ( g_m\bar n_1 n_1 a_1  M, \wl.\mu)\not =0$$
implies
$n_1a_1=n_0a_0 g $, where $g\in NA$ is $\eps$-close to identity. \end{lem}
\begin{proof} Just note from the expression of $B^{(T)}_M\circ\gamma_2  $ that, if it is not $0$, we must have
$$g_m\bar n_1n_1 a_1 \in \supp Q_{\omega_0}.$$
The element $g_m$ varies in a finite set and $\bar n_1$ varies over $\Omega$ which is of diameter $\leq \eps$. We also assume that $\supp Q_{\omega_0}$ is of diameter $\leq \eps$, so that $n_1$ (and $a_1$) must both vary in sets of diameter $\leq \eps$.
\end{proof}
It follows that $\bar n_1 n_1 a_1M$ itself is $\eps$-close to
$n_0a_0 M$ in $G/M$.
From now on we write $g_m\bar n_1 n_1 a_1M= g_m n_0 a_0 gM$,
where $gM\in G/M$ varies in a neighbourhood of $eM$ of diameter $\leq \eps$.
We will always choose a representative
$g\in \exp(\fN\oplus \fA\oplus \bar \fN)$. By $G$-equivariance we may
assume $g_m n_0 a_0=1$, which we do from now on.

\begin{prop}\label{p:asymp} (Contracting and expanding foliations)
\begin{enumerate}
\item Let $\mu$ be such that $\alpha_k(X_{\mu})> 0$
for all $\alpha_k\in \Delta^+$ with $k>j_0$ (this is of course the case if $\mu$ is close enough to $\Lambda_{\infty}$).
Suppose we have $gM$ and $g'M$ both $\eps$-close to $eM$ such that
$ B_M^{(T)} \!\!\circ \!\gamma_2\,\,\,\, ( g M, \wl.\mu) \neq 0$
and
$ B_M^{(T)} \!\!\circ \!\gamma_2\,\,\,\, ( g'M, \wl.\mu) \neq 0$,
then $g^{'-1}g=\exp(X+\sum_{\alpha\in \Delta^+}Y_\alpha+\sum_{\alpha\in \Delta^+}Y_{\wl.\alpha})$ with $X\in \fA, Y_\alpha\in\fG_\alpha$, $\norm{X},
\norm{Y_\alpha}\leq \eps$, and $\norm{Y_{\wl.\alpha_k}}\leq \eps e^{-T\eta(\wl.\alpha_k)(X_{\wl.\mu})}=
\eps e^{-T\eta \alpha_k(X_{\mu})}$ for $k>j_0$.

\item Similarly, assume that $\alpha_k(X_\lambda)>0$ for all $\alpha_k\in\Delta^+$ with $k>j_0$.
 Suppose we have $gM$ and $g'M$ both $\eps$-close to $eM$ such that
    $A^{(T)}_M\circ\gamma_1( g \wl M, \lambda) \neq 0$
and $A^{(T)}_M\circ\gamma_1( g'\wl M, \lambda) \neq 0$.
Then $g^{'-1}g=\exp(X+\sum_{\alpha}Y_\alpha)$ with
$X\in \fA, Y_\alpha\in\fG_\alpha$, $\norm{X}, \norm{Y_\alpha}\leq \eps$,
$\norm{Y_{\alpha_k}}\leq \eps e^{-T\eta\alpha_k(X_{\lambda})}$ for $k>j_0$.
\end{enumerate}
\end{prop} 
 Actually, the claim holds for all $k$ (not only for $k>j_0$), but we will only use it for $k>j_0$.
 For the other indices, there is something more optimal to do.
\begin{proof}
Assume that the term $B_M^{(T)} \!\!\circ \!\gamma_2\,(gM, \wl.\mu)$
does not vanish.  The evolution equation \eqref{e:conditions} shows
that we must have\footnote{Here the $Q_{\omega}$ are treated as functions on $G/M$ that factor through $G/K$.}
\begin{itemize}
\item $ge^{-(T-1) \eta X_{\wl.\mu} }M \in \gamma_2^{-1}.\supp Q_{\omega_{T-1}}$;
\item $g M \in \supp Q_{\omega_0}$.
\end{itemize}

If $gM$ and $g'M$ both satisfy the two conditions above, then we see that
$g'^{-1}g$ must be $\eps$-close to identity.  For $\eps$ small enough
we can write this element using the co-ordinates described in part (1) of the claim.
Also, $e^{(T-1) \eta X_{\wl.\mu} } g'^{-1}ge^{-(T-1) \eta X_{\wl.\mu} }$
must stay in the fixed compact set
$$M [\supp Q_{\omega_{T-1}}]^{-1} \supp Q_{\omega_{T-1}} M \subset G.$$
Writing the action of $A$ in the co-ordinate system gives the claim.
The proof of the second part is similar.
\end{proof}

Finally we write $gM=\bar{n} n a M$ with $\bar n\in\overline{N},n\in N, a\in A$ all $\eps$-close to $1$.
We decompose $n=e^Y n_\textrm{fast}$, and $\bar n=e^{\bar{Y}} \bar{n}_\textrm{fast}$,
$Y\in \lns\simeq \IR^{J_0}$, $\overline{ Y}\in \lnbs\simeq \IR^{J_0}$ both $\eps$-close to $0$
 (we fix a vector space isomorphism that sends the root spaces
to the coordinate axes of $\IR^{J_0}$); and $n_\textrm{fast}\in \Nf$, $\bar{n}_\textrm{fast}\in \Nbf$ both $\eps$-close to $1$.  The quantity $\epsilon$ is fixed, but can be chosen as small as we wish. Note that the previous Proposition restricts $n_\textrm{fast}$ and $\bar n_\textrm{fast}$ to sets of measure $\prod_{k>j_0} \eps e^{-T\eta m_{\alpha_k}\alpha_k(X_{\lambda})}$ and
 $\prod_{k>j_0} \eps e^{-T\eta  m_{\alpha_k}\alpha_k(X_{\mu})}$, respectively.

We will now break $\cP_m$ into countably many pieces,
$$\cP_{m}=\sum_{(\bar y, y, t,\lambda_0)\in\IZ^{2J_0}\times \IZ^{2r}} \cP_{m, (\bar y, y,t, \lambda_0)}$$
to which we will apply the Cotlar-Stein lemma.

For $j=J_0$ and $j=r$ choose a smooth nonnegative compactly supported function $\chi^j$ on $\IR^{ j}$ such that \begin{equation}\sum_{y\in \IZ^j }\chi^j(Y-y)\equiv 1
\label{e:decompcs}
\end{equation}
and such that $\chi^j(Y).\chi^j(Y+2y)=0$ for all $Y\in \IR^j$
and $y\in \IZ^j\setminus \{0\}.$

Let $(\bar y, y)\in \IZ^{2J_0}$ and let $(t,\lambda_0)\in \IZ^{2r}$. Denote $2^+$ a fixed real number $>2$.
Define $\chi^\hbar_{(\bar y, y )}(\overline{ Y}, Y)=\chi^{J_0}(\hbar^{-1/2^+}\overline{ Y}-\bar y)\chi^{J_0}(\hbar^{-1/2^+} Y-y)$; 
and $\chi^\hbar_{\lambda_0}(\lambda)=\chi^r(\hbar^{-1/2^+}\lambda-\lambda_0)$ and
$\chi^\hbar_{t}(a)=\chi^r(\hbar^{-1/2^+ }a-t)$.
Also define $\chi^\hbar_{(\bar y, y, t)}(gM)=\chi^\hbar_{(\bar y, y, t )}(\overline{ Y}, Y) \chi^\hbar_{t}(a)$ if $gM$ is an element
of $G/M$ that can be decomposed as $gM= e^{\overline{ Y}}\bar n_\textrm{fast}  e^Y  n_\textrm{fast} a M$, as described above.

We define a bounded operator $\cS_{m, (\bar y, y, t, \lambda_0)}:L^2(G/K)\To L^2(G/K)$ by
$$ \cS_{m, (\bar y, y,t, \lambda_0)}\left[  e_{i\hbar^{-1}\wl.\mu, k}\right] (x)
\defeq 
e^{ \frac{-iT\eta H(\mu)}\hbar}e_{i\hbar^{-1}\wl.\mu, k}(x)\,\, \chi^\hbar_{(\bar y, y, t)}(x, k)\,
\chi^\hbar_{\lambda_0}(\mu)
\,B_M^{(T)} \!\!\circ \!\gamma_2\,\, ( x, k, \wl.\mu ).$$
%Applying the (non)stationary phase lemma, one shows `easily' that
%\begin{equation}\label{e:Q'}\cS_{m, (\bar y, y, t, \lambda_0)}\left[  e_{i\hbar^{-1}\wl.\mu, k}\right]= \cS_{m, (\bar y, y, t, \lambda_0)} Q'_{\omega_{T-1}}\circ\gamma_2 \left[  e_{i\hbar^{-1}\wl.\mu, k}\right]+\cO(\hbar^\infty) 
%\end{equation}
%(this is the only place where we need $2^+>2$, otherwise we could as well take $2^+=2\spadesuit$.)
We then define
$$\cP_{m, (\bar y, y, t, \lambda_0)}\defeq \Pi_\Gamma Q'_{\omega_{T-1}}\circ \gamma_2\,\, \cS^*_{m, (\bar y, y, t, \lambda_0)}\cU_\chi Q^{'2}_{\omega_{-T}}\circ\gamma_1.$$
It can be checked that
$$\norm{\cP_{m}-\sum_{(\bar y, y, t, \lambda_0)\in\IZ^{2k_o+2r}} \cP_{m, (\bar y, y, t, \lambda_0)}}_{L^2(\bY)\To L^2(\bY)}=\cO(\hbar^{M-d/2}),$$
by noting that the sum $ \sum_{(\bar y, y, t, \lambda_0)\in\IZ^{2k_o+2r}} \cS_{m, (\bar y, y, t, \lambda_0)}$ gives back our Ansatz \eqref{e:ansatzforS} for $Q_m \cS$, and by arguing as in \eqref{e:oneofthethings} that the difference between $Q_m \cS$ and the Ansatz is of order $\cO(\hbar^{M-d/2})$. Again we choose $M$ large enough so that the error $\cO(\hbar^{M-d/2})$ is negligible compared to the bound announced in Theorem \ref{t:main}.

 \vspace{.8cm}
 Let us now look at a scalar product $\left\la \cS_{m, (\bar y, y, t, \lambda_0)}   e_{i\hbar^{-1}\wl.\mu, \bar n P_0},\,\,
 \cU_\chi   Q^{'}_{\omega_{-T}}\circ\gamma_1\,e_{i\hbar^{-1}\lambda, \theta}\right\ra$.
We need only consider the generic case where $\theta \in \bar n P_0 \wl P_0$, that is, $\theta$ is of the form $\theta=\bar n n\wl P_0$ (with $n\in N$). From the previous discussions, it follows that this scalar product is non-negligible only if $\bar n $ and $n$ stay in some sets of diameters $\leq \eps$; and, without loss of generality, we have assumed they are both $\eps$-close to $1$. 
As in \eqref{e:development}, we have by the stationary phase method
\begin{multline}\label{e:astep}
\left\la \cS_{m, (\bar y, y, t, \lambda_0)}   e_{i\hbar^{-1}\wl.\mu, \bar n P_0},\,\,
 \cU_\chi    Q^{'}_{\omega_{-T}}\circ\gamma_1\,e_{i\hbar^{-1}\lambda, \theta}\right\ra
 \\
=\hbar^{J/2}\int_{a\in A } d(\lambda, a) \,C^{(\bar y, y, t, \lambda_0)}_\hbar\left(\bar n n a  M , \lambda, \wl.\mu\right) 
\bar e_{i\hbar^{-1}\wl.\mu, \bar n P_0} (\bar n n  aK)
e_{i\hbar^{-1}\lambda, \bar n n wP_0}(\bar n n   aK)
da\\
=\hbar^{J/2}\int_{a\in A } d(\lambda, a) \,C^{(\bar y, y, t, \lambda_0)}_\hbar\left(\bar n n   aM , \lambda, \wl.\mu\right) 
\bar e_{i\hbar^{-1}\wl.\mu, \bar nP_0 } (\bar n n   aK)
e_{i\hbar^{-1}\lambda, \bar n n \wl P_0}(\bar n n  aK)
 da.
\end{multline}
where  $C^{(\bar y, y, t, \lambda_0)}_\hbar\left(\bar n n   aM, \lambda, \wl.\mu \right) =\sum \hbar^k c_k\left(\bar n n   aM, \lambda, \wl.\mu \right) $
and \begin{equation*}c_0\left(\bar n n   aK,
 \bar n P_0,  \bar nn \wl P_0, \lambda, \wl.\mu \right) =
A^{(T)}_M\circ\gamma_1(\bar n n a \wl M, \lambda)\bar B^{(T)}_M\circ\gamma_2(\bar n n a M,\wl. \mu) \times
\chi^\hbar_{(\bar y, y, t )}(\bar n n aM)\chi^\hbar_{\lambda_0}(\mu)
\end{equation*}
(the next terms have the same support as the leading one). Remember the notation $\bar n=e^{\overline{ Y}}\bar n_{fast}$, $n=e^Y n_{fast}$. By Proposition \ref{p:asymp}, and by definition of the cut-off functions $\chi^{J_0}, \chi^r$, our scalar product is non-negligible only if $\overline{ Y}, Y$ stay in a set of measure $\hbar^{J_0/2^+}$, and $n_{fast}, \bar n_{fast}$ stay in a set of measure $\prod_{k>j_0}  e^{-T\eta m_{\alpha_k}\alpha_k(X_{\lambda})}$ and
 $\prod_{k>j_0}  e^{-T\eta  m_{\alpha_k}\alpha_k(X_{\mu})}$, respectively.

 %%%%%%%%%%%%%%%%%%%%%%%%%%%%%%%%%%%%%%%%%%%%%%%%%%%
\subsection{Norm of $\cP^*_{m, (\bar x, x, s, \mu_0) }\cP_{m, (\bar y, y, t, \lambda_0)}$}
%%%%%%%%%%%%%%%%%%%%%%%%%%%%%%%%%%%%%%%%%%%%%%%%%%
We are now ready to check the first assumption of the Cotlar-Stein lemma, that is, to
bound from above the norm of $\cP^*_{m, (\bar x, x, s, \mu_0) }\cP_{m, (\bar y, y, t, \lambda_0)}$.

Let $u, v\in L^2(\Gamma\backslash G/K)$. We write
\begin{multline*}
\left\la \cP_{m, (\bar x, x, s, \mu_0) } v, \,\cP_{m, (\bar y, y, t, \lambda_0)}u\right\ra_{\Gamma\backslash G/K} \\ =\left\la  Q'_{\omega_{T-1}}\circ\gamma_2\cS^*_{m, (\bar x, x, s, \mu_0) } \cU_\chi Q^{'2}_{\omega_{-T}}\circ\gamma_1\, v, \,Q'_{\omega_{T-1}}\circ\gamma_2
 \cS^*_{m,  (\bar y, y, t, \lambda_0)}  \cU_\chi  Q^{'2}_{\omega_{-T}}\circ\gamma_1\,u\right\ra_{G/K} 
  \\
  =\left\la   \cS^*_{m, (\bar x, x, s, \mu_0) } \cU_\chi Q^{'2}_{\omega_{-T}}\circ\gamma_1\, v,  \, \cS^*_{m, (\bar y, y, t, \lambda_0)}  \cU_\chi  Q^{'2}_{\omega_{-T}}\circ\gamma_1\,u\right\ra_{G/K} 
 + \cO(\hbar^{\infty})\norm{u}
 \norm{v}.
 \end{multline*}
   
 We develop fully this scalar product using the Fourier transform.
 \begin{multline} 
\left\la \cS^*_{m,  (\bar x, x, s, \mu_0) } \cU_\chi Q^{'2}_{\omega_{-T}}\circ\gamma_1 \,v,   \cS^*_{m,  (\bar y, y, t, \lambda_0)}  \cU_\chi  Q^{'2}_{\omega_{-T}}\circ\gamma_1\,u\right\ra_{G/K}  
  \\
  =\int  d\theta d\theta'
  |c_\hbar(\lambda)|^{-2}d\lambda |c_\hbar(\lambda')|^{-2}d\lambda' \widehat{Q^{'}_{\omega_{-T}}\circ\gamma_1\, u}(\lambda, \theta)\overline{\widehat{Q^{'}_{\omega_{-T}}\circ\gamma_1\, v}(\lambda', \theta')}
 \\ 
\left\la \cS^*_{m,  (\bar x, x, s, \mu_0) }  \cU_\chi Q^{'}_{\omega_{-T}}\circ\gamma_1 e_{i\hbar^{-1}\lambda', \theta'},
  \cS^*_{m,  (\bar y, y, t, \lambda_0)} \cU_\chi Q^{'}_{\omega_{-T}}\circ\gamma_1e_{i\hbar^{-1}\lambda, \theta}\right\ra_{G/K}
\\
  =\int d\theta d\theta' dk
  |c_\hbar(\lambda)|^{-2}d\lambda |c_\hbar(\lambda')|^{-2}d\lambda' |c_\hbar(\mu)|^{-2}d\mu
   \widehat{Q^{'}_{\omega_{-T}}\circ\gamma_1\, u}(\lambda, \theta)\overline{\widehat{Q^{'}_{\omega_{-T}}\circ\gamma_1\, v}(\lambda', \theta')}
 \\  \left\la \cU_\chi Q^{'}_{\omega_{-T}}\circ\gamma_1 e_{i\hbar^{-1}\lambda', \theta'}, \cS_{m, (\bar x, x, s, \mu_0) } 
 e_{i\hbar^{-1}\wl.\mu, k}\right\ra \left\la  \cS_{m, (\bar y, y, t, \lambda_0)} e_{i\hbar^{-1}\wl.\mu, k}
  \cU_\chi Q^{'}_{\omega_{-T}}\circ\gamma_1e_{i\hbar^{-1}\lambda, \theta}\right\ra_{G/K}
   \\
  =
  \int d\theta d\theta' \Jac(\bar n) d\bar n
  |c_\hbar(\lambda)|^{-2}d\lambda |c_\hbar(\lambda')|^{-2}d\lambda' |c_\hbar(\mu)|^{-2}d\mu
   \widehat{Q^{'}_{\omega_{-T}}\circ\gamma_1\, u}(\lambda, \theta)\overline{\widehat{Q^{'}_{\omega_{-T}}\circ\gamma_1\, v}(\lambda', \theta')}
 \\  \left\la \cU_\chi Q^{'}_{\omega_{-T}}\circ\gamma_1 e_{i\hbar^{-1}\lambda', \theta'}, \cS_{m, (\bar x, x, s, \mu_0) } 
 e_{i\hbar^{-1}\wl.\mu, \bar nP_0}\right\ra \left\la  \cS_{m, (\bar y, y, t, \lambda_0)} e_{i\hbar^{-1}\wl.\mu, \bar nP_0},\,
  \cU_\chi Q^{'}_{\omega_{-T}}\circ\gamma_1e_{i\hbar^{-1}\lambda, \theta}\right\ra_{G/K}
 \label{e:last}
  \end{multline}

Finally, in equation \eqref{e:last}, we write $\theta=\bar n n\wl P_0$ and $\theta'=\bar n n'\wl P_0$
(we can do so on a set of full measure). We have shown in \eqref{e:astep} that
\begin{multline}  
 \left\la \cU_\chi Q^{'}_{\omega_{-T}}\circ\gamma_1 e_{i\hbar^{-1}\lambda', \theta'}, \cS_{m, (\bar x, x, s, \mu_0) } 
 e_{i\hbar^{-1}\wl.\mu, \bar nP_0}\right\ra \left\la  \cS_{m, (\bar y, y, t, \lambda_0)} e_{i\hbar^{-1}\wl.\mu, \bar nP_0},\,
  \cU_\chi Q^{'}_{\omega_{-T}}\circ\gamma_1e_{i\hbar^{-1}\lambda, \theta}\right\ra_{G/K}
  \\=
  \hbar^{J}\int_{a\in A } d(\lambda,a)\, C^{(\bar y, y, t, \lambda_0)}_\hbar\left(\bar n n   a M , \lambda, \wl.\mu\right) 
\bar e_{i\hbar^{-1}\wl.\mu, \bar n } (\bar n n   aK)
e_{i\hbar^{-1}\lambda, \bar n n w}(\bar n n   aK)
 da\\
\int_{a'\in A } d(\lambda', a')\,\bar C^{(\bar x, x, s, \mu_0) }_\hbar\left(\bar n n'    a'M, \lambda', \wl.\mu\right) 
 e_{i\hbar^{-1}\wl.\mu, \bar n } (\bar n n'   a'K)
\bar e_{i\hbar^{-1}\lambda', \bar n n' w}(\bar n n'    a'K) da'\label{e:almostend}.
 \end{multline}
  Already we can note that $C^{(\bar y, y, t, \lambda_0)}_\hbar\left(\bar n n   aM, \lambda, \wl.\mu \right)  \bar C^{(\bar x, x, s, \mu_0) }_\hbar\left(\bar n n'  a'M, \lambda', \wl.\mu\right)$
 can only be non zero if $\chi^\hbar_{(\bar y, y, t )}(\bar nn   a M) \chi^\hbar_{(\bar x, x, s ) }(\bar nn'   a' M) \not=0$, and from the way we chose $\chi^{J_0}$ this can happen only for $\norm{ \bar x-\bar y}\leq 2$.
 For the same reason, it can only be non zero if $\norm{\mu_0-\lambda_0}\leq 2$. 
  
  Now we try to show that \eqref{e:last} decays fast when $\norm{x-y}$
  gets large.
Under the last integral in \eqref{e:last} we have a function of the pair $(\bar n n a, \bar n n' a')$. 
We have an oscillatory integral of the form \eqref{e:NSP},
with a phase
\begin{eqnarray*}S(\bar n n a, \bar n n' a')&=&\lambda. B(\bar n n a \wl M)+(\wl.\mu)[B(\bar n n' a')-B(\bar n n a)]-\lambda'.
B(\bar n n' a' \wl M)\\
&=&\lambda. B(\bar n n a \wl M)+(\wl.\mu)[ a'-a]-\lambda'.
B(\bar n n' a' \wl M),
\end{eqnarray*}
where $B$ is the function defined in \eqref{e:buse}.
We want to do ``integration by parts with respect to $\bar n$'' (as in Lemma \ref{l:NSP}). However, because the derivatives
of $S$ with respect to $\bar n$ are tricky to compute, it is preferable to use a vector field $Z$ whose definition is a bit delicate but with the property that $Z. B(\bar n n a \wl M)=0$ and $Z.B(\bar n n' a' \wl M)=0$.

Consider a variation of the form
$$\Psi^\tau:(\bar n na, \bar n n'a')\mapsto (\bar n n  e^{\tau Y}a, \bar n n'  a'a^{-1} e^{\Psi(\tau Y)}  a)=
\bar n n (e^{\tau Y}a,   n^{-1}n' a'a^{-1} e^{\Psi(\tau Y)}  a),$$
for $Y\in\bar\fN$, and $\Psi=\Psi_{ n^{-1}n' a'a^{-1}}$ defined in lemma \ref{l:n-}.
By definition of $\Psi$, the two elements $\bar n ne^{\tau Y}a$ and $\bar n n'a'a^{-1} e^{\Psi(\tau Y)}  a$
are in the same $NA$ orbit, for all $\tau$. Such a variation preserves the terms
$B(\bar n n' a' \wl M)$ and $B(\bar n n a \wl M)$.
We call $Z$ the vector field $\frac{d\Psi^\tau}{d\tau}_{|\tau=0}$. We take $Y\in \fG_{\wl.\alpha_k}$ with $1\leq k\leq j_0$.
We note that each term of the product
\begin{multline*}  \widehat{Q^{'}_{\omega_{-T}}\circ\gamma_1\, u}(\lambda, \bar n n \wl P_0)\overline{\widehat{Q^{'}_{\omega_{-T}}\circ\gamma_1\, v}(\lambda', \bar n n' \wl P_0)}
e_{i\hbar^{-1}\lambda, \bar n n \wl}(\bar n n   aK) \bar e_{i\hbar^{-1}\lambda', \bar n n' \wl}(\bar n n'    a'K)
 \end{multline*} is invariant under $\Psi^\tau$.
The function $C^{(\bar y, y, t, \lambda_0)}_\hbar\left(\bar n n   aM \right)  \bar C^{(\bar x, x, s, \mu_0) }_\hbar\left(\bar n n'   aM \right) $ satisfies
$$\norm{Z^mC^{(\bar y, y, t, \lambda_0)}_\hbar\left(\bar n n   aM \right)  \bar C^{(\bar x, x, s, \mu_0) }_\hbar\left(\bar n n'    a'M\right)}\leq C(m)\hbar^{-m/2^+}$$
just by the definition of $C^{(\bar y, y, t, \lambda_0)}_\hbar$ and $C^{(\bar x, x, s, \mu_0) }_\hbar$. Now we want to apply the nonstationary phase lemma \ref{l:NSP}, so we need to understand $ZS= Z\left[(\wl.\mu)[B(\bar n n' a')-B(\bar n n a)]\right].$ 
 
 Lemmas \ref{l:n-} and \ref{l:n?} tell us that if we write $n^{-1}n' = \exp(T)$
with $T = \sum_\alpha T_\alpha$ $\epsilon$-close to $0$, choose $\beta$ among the slow exponents
so that the norm of $T_\beta$ is comparable to the norm of $\log(n^{-1}n')_\textrm{slow}$ and take $Y\in\lnbs$ of norm $1$ such that
$\Psi_0'(Y) = \theta(T_\beta)$ then
\begin{equation}|ZS(\bar n na, \bar n n'a')|\geq C\norm{\log(n^{-1}n')_{\mathrm{slow}}}.\end{equation}
 Note that we have $\norm{\log(n^{-1}n')_{\mathrm{slow}}}\geq\hbar^{1/2^{+}}(\norm{x-y}-4)$
if $C^{(\bar y, y, t, \lambda_0)}_\hbar\left(\bar n n   aM \right)  \bar C^{(\bar x, x, s, \mu_0) }_\hbar\left(\bar n n'   aM \right) \not= 0$.
 
 We now apply Lemma \ref{l:NSP} to the last expression of integral \eqref{e:last}, integrating by parts $\tilde M$-times using the vector field $Z$.
 This
 yields
  that $\left\la \cP_{m, (\bar x, x, s, \mu_0) } v, \cP_{m, (\bar y, y, t, \lambda_0)}u\right\ra_{\bY}$ is bounded from above by
 \begin{multline*}\frac{C(M)\hbar^{\tilde M(1-2/2^+)}}{\max(16,\norm{x-y})^{\tilde M} }
  \hbar^{J}\\
  \int \Jac(n)dn \Jac(n')dn' \Jac(\bar n)d\bar n da\, da' \chi^\hbar_{(\bar y, y, t )}(\bar nn   a M) \chi^\hbar_{(\bar x, x, s ) }(\bar nn'   a' M) 
  |c_\hbar(\lambda)|^{-2}d\lambda |c_\hbar(\lambda')|^{-2}d\lambda' |c_\hbar(\mu)|^{-2}d\mu\\
\chi_{\lambda_0}^\hbar(\mu)\chi_{\mu_0}^\hbar(\mu) | \widehat{Q^{'}_{\omega_{-T}}\circ\gamma_1\, u}(\lambda, \bar n n\wl P_0)\widehat{Q^{'}_{\omega_{-T}}\circ\gamma_1\, v}(\lambda', \bar n n' \wl P_0)|
 \end{multline*}
 for an arbitrarily large integer $\tilde M$. For any $\bar n, n, n'$, we have
$$\int da\, da' \chi^\hbar_{(\bar y, y, t )}(\bar nn   a M) \chi^\hbar_{(\bar x, x, s ) }(\bar nn'   a' M) =\cO(\hbar^{2r/2^+}),$$
so the previous bound becomes
 \begin{multline*}\frac{\hbar^{\tilde M(1-2/2^+)}}{\max(16,\norm{x-y})^{\tilde M} }
  \hbar^{J} \hbar^{2r/2^+}\\
  \int \Jac(n)dn \Jac(n')dn' \Jac(\bar n)d\bar n  
  |c_\hbar(\lambda)|^{-2}d\lambda |c_\hbar(\lambda')|^{-2}d\lambda' |c_\hbar(\mu)|^{-2}d\mu\\
\chi_{\lambda_0}^\hbar(\mu)\chi_{\mu_0}^\hbar(\mu) | \widehat{Q^{'}_{\omega_{-T}}\circ\gamma_1\, u}(\lambda, \bar n n\wl P_0)\widehat{Q^{'}_{\omega_{-T}}\circ\gamma_1\, v}(\lambda', \bar n n' \wl P_0)|
 \end{multline*}

 Similarly, $\tilde M$ integrations by parts in \eqref{e:last} with respect to the variable $\mu$ allows to gain a factor 
 $\frac{\hbar^{\tilde M(1-1/2^+)}}{\norm{a-a'}^{\tilde M}}\leq \frac{\hbar^{\tilde M(1-2/2^+)}}{\norm{t-s}^{\tilde M} }$ if $\norm{t-s}$ is large enough.
 Integrations by parts with respect to $a$ allow to gain a factor
 $\frac{\hbar^{\tilde M(1-1/2^+)}}{\norm{\lambda-\mu}^{\tilde M}}$; and integrations by parts with respect to $a'$ allow to gain a factor
 $\frac{\hbar^{\tilde M(1-1/2^+)}}{\norm{\lambda'-\mu}^{\tilde M}}$. In particular, the contribution to \eqref{e:last} 
 of those $\lambda,\lambda', \mu$ with $\norm{\lambda'-\mu}\geq \hbar^{1/2}$ or $\norm{\lambda-\mu}\geq h^{1/2}$ is $\cO(\hbar^\infty)$. In these cases the application of the non-stationary phase lemma \ref{l:NSP}
 is made simpler by the fact that the phase $S$ is linear in $\mu$,  $a$ and $a'$.

 We find that $\left\la \cP_{m, (\bar x, x, s, \mu_0) } v, \cP_{m, (\bar y, y, t, \lambda_0)}u\right\ra_{\bY}$ is bounded from above by
 \begin{multline}\frac{1}{ \max(16,\norm{x-y})^{\tilde M}}\frac{1}{\max(16,\norm{t-s})^{\tilde M }}
  \hbar^{J}\hbar^{2r/2^{+}} \\
  \int \Jac(n)dn \Jac(n')dn' \Jac(\bar n)d\bar n
  |c_\hbar(\lambda)|^{-2}d\lambda |c_\hbar(\lambda')|^{-2}d\lambda' |c_\hbar(\mu)|^{-2}d\mu\\
\chi_{\lambda_0}^\hbar(\mu)\chi_{\mu_0}^\hbar(\mu) | \widehat{Q^{'}_{\omega_{-T}}\circ\gamma_1\, u}(\lambda, \bar n n\wl P_0)\widehat{Q^{'}_{\omega_{-T}}\circ\gamma_1\, v}(\lambda', \bar n n' \wl P_0)|.
 \end{multline}
  In this integral, $\lambda', \lambda, \mu$ are all $\eps$-close to $\Lambda_\infty$, and each of them runs
  over a set of volume $\hbar^{r/2^+}$; $\bar n$ runs over a set of measure $\hbar^{J_0/2^+}
 \prod_{k>j_0}  e^{-T\eta m_{\alpha_k}\alpha_k(X_{\mu})}$, $n$ runs over a set of measure $\hbar^{J_0/2^+}\prod_{k>j_0}  e^{-T\eta m_{\alpha_k}\alpha_k(X_{\lambda})}$, and $n'$ runs over a set of measure $\hbar^{J_0/2^+}\prod_{k>j_0}  e^{-T\eta m_{\alpha_k}\alpha_k(X_{\lambda'})}$.
  Using Cauchy-Schwarz and the Plancherel formula we find that the integral  
$$ 
\int \Jac(n)dn  \Jac(n')dn'  
  |c_\hbar(\lambda)|^{-2}d\lambda |c_\hbar(\lambda')|^{-2}d\lambda'  
 | \widehat{Q^{'}_{\omega_{-T}}\circ\gamma_1\, u}(\lambda, \bar n n\wl P_0)\widehat{Q^{'}_{\omega_{-T}}\circ\gamma_1\, v}(\lambda', \bar n n' \wl P_0)|
$$
is bounded by $\hbar^{-d}\hbar^{J_0/2^+}\hbar^{r/2^{+}}
\prod_{k>j_0}  e^{-T\eta m_{\alpha_k}\alpha_k(X_{\Lambda_\infty)} }\hbar^{-J\cK \epsilon}
\norm{u}_{L^2(\bY)}\norm{v}_{L^2(\bY)} $.

 The integral 
$\int  \Jac(\bar n)d\bar n
   |c_\hbar(\mu)|^{-2}d\mu$ adds another factor $\hbar^{-d}\hbar^{J_0/2^+}\hbar^{r/2^{+}}
   \prod_{k>j_0}  e^{-T\eta m_{\alpha_k}\alpha_k(X_{\Lambda_\infty)} }\hbar^{-J\cK \epsilon}$.
   Overall we find that
   \begin{multline*}\norm{\cP^*_{m, (\bar x, x, s, \mu_0) }\cP_{m, (\bar y, y, t, \lambda_0)}}
  \\ \leq\frac{1}{ \max(16,\norm{x-y})^{\tilde M}}\frac{1}{\max(16,\norm{t-s})^{\tilde M} }
    \hbar^{J+ 4r/2^{+}-2d+2J_0/2^+} \\
  \prod_{k>j_0}  e^{-2T\eta m_{\alpha_k}\alpha_k(X_{\Lambda_\infty)} }\hbar^{-2J\cK \epsilon}
    \end{multline*}
  and vanishes for  $\norm{ \bar x-\bar y}>2$ or
 $\norm{\mu_0-\lambda_0}> 2$. 
 
    Choosing $\tilde M$ large enough, we can sum over all $  (\bar y, y, t, \lambda_0)$, and we find
    \begin{equation*}\sum_{  (\bar y, y, t, \lambda_0)\in \IZ^{2J_0+2r}} 
   \norm{\cP^*_{m, (\bar x, x, s, \mu_0) }\cP_{m, (\bar y, y, t, \lambda_0)}}^{1/2}\\
\leq
 \hbar^{J/2+ 2r/2^{+}-d+J_0/2^+}\prod_{k>j_0}  e^{-T\eta m_{\alpha_k}\alpha_k(X_{\Lambda_\infty)} }\hbar^{-J\cK \epsilon}.
 \end{equation*}
 Remembering that $J=d-r$ and that $2^+$ could be chosen arbitrarily close to $2$, we get
  $$\sum_{  (\bar y, y, t, \lambda_0)\in \IZ^{2J_0+2r}} 
   \norm{\cP^*_{m, (\bar x, x, s, \mu_0) }\cP_{m, (\bar y, y, t, \lambda_0)}}^{1/2}
\leq
 \hbar^{\frac{J_0-J}2-c\eps} \prod_{k>j_0}  e^{-T\eta m_{\alpha_k}\alpha_k(X_{\Lambda_\infty})}$$
 with a constant $c$ that depends on $\cK$.

\subsection{Norm of $\cP_{m, (\bar x, x, s, \mu_0)}\cP^*_{m, (\bar y, y, t, \lambda_0)}$}
 Using a similar calculation reversing the roles of $\overline{N}$ and $N$, we get the same bound,
  $$\sum_{  (\bar y, y, t, \lambda_0)\in \IZ^{2J_0+2r}} 
   \norm{\cP_{m, (\bar x, x, s, \mu_0) }\cP^*_{m, (\bar y, y, t, \lambda_0)}}^{1/2}
\leq
 \hbar^{\frac{J_0-J}2-c\eps} \prod_{k>j_0}  e^{-T\eta m_{\alpha_k}\alpha_k(X_{\Lambda_\infty})}.$$

Using the Cotlar-Stein lemma and the fact that the $\alpha(X_{\Lambda_\infty})$
coincide with the Lyapunov exponents $\chi(H)$ on the energy layer $\cE_{\Lambda_\infty}$,
we get
Theorem \ref{t:main}.

\section{Measure Rigidity\label{s:rigid}}
In this section we prove Theorems \ref{t:coroSL3}, \ref{t:coroSL4} and
\ref{t:coroinner}.  The proofs combine our entropy bounds
with the measure classification results of~\cite{EK03, EKL06} and the
orbit classification results of~\cite{LW01, Tom00} which give information
about $A$-invariant and ergodic measures that have a large entropy.

\begin{prop}\label{p:rigid}(Measure rigidity theory)
Let $G$ be a split group, and let $\mu$ be an ergodic $A$-invariant measure on
$X=\Gamma\backslash G$.

(1) \cite[Lem.\ 6.2]{EK03} there exist constants $s_\alpha(\mu)\in [0, 1]$
associated to the roots $\alpha\in\Delta$, such that for any $a\in A$,
$$h_{KS}(\mu, a)=\sum_{\alpha\in\Delta}s_\alpha(\mu)\left(\log \alpha(a)\right)^+.$$
Here $t^+=\max\{0, t\}$ for $t\in\IR$. Furthermore, $s_\alpha(\mu)=1$ if and only if $\mu$ is invariant by the root subgroup $U_\alpha$.

(2) \cite[Prop.\ 7.1]{EK03} Assume that $s_\alpha(\mu), s_\beta(\mu)>0$ for two roots $\alpha, \beta\in\Delta$ such that $\alpha+\beta\in\Delta$. Then $s_{\alpha+\beta}(\mu)=1$.

(3) \cite[Thm.\ 4.1(iv)]{EK03} If $G$ is locally isomorphic to $SL_n$ and $s_\alpha(\mu)>0$ for all $\alpha$, then $\mu$ is $G$-invariant.

(4) \cite[Cor.\ 3.4]{EKL06} In the case $G=SL_n$, we have $s_\alpha(\mu)=s_{-\alpha}(\mu)$ for all roots $\alpha$.
\end{prop}

We do not know if (4) holds in general.

Now let $\mu$ be an $A$-invariant probability measure with ergodic
decomposition $\mu=\int_X \mu_x d\mu(x).$
For each subset $R\subset \Delta$ let $X_R$ be the set of ergodic components
$\mu_x$ such that $\{\alpha, s_\alpha(\mu_x)>0\}=R$.
Write $w_R=\mu(X_R)$ and if $w_R>0$, let
$\mu_R=\frac{1}{w_R}\int_{X_R}\mu_x d\mu(x)$, so that $\mu=\sum_{R}w_R\mu_R$.
From Proposition \ref{p:rigid}(1) we have for $a\in A$
$$h_{KS}(\mu_R, a)\leq \sum_{\alpha\in R}\left(\log \alpha(a)\right)^+,$$
(this is in fact an avatar of the Ruelle-Pesin inequality) and hence
$$h_{KS}(\mu, a)\leq \sum_Rw_R\sum_{\alpha\in R}\left(\log \alpha(a)\right)^+.$$
By Proposition \ref{p:rigid}(2) it is enough to consider those $R$ that are
closed under the addition of roots.  In the case $G=SL_n$, parts (3) and (4)
show, respectively, that it is enough to consider those $R$ which are
symmetric and that also $\mu_\Delta=\mu_{Haar}$.

\begin{prop}\label{p:SL3} Let $G=SL_3(\IR)$, $\Gamma$ a lattice in $G$, and $\mu$ an $A$-invariant probability measure on $\Gamma\backslash G$, such that
$h_{KS}(\mu, a)\geq \frac12 h_{KS}(\mu_{Haar}, a)$ for $a=e^X$, $X=\diag(2, -1,-1)$, $\diag(-1, 2-1)$, and $\diag( -1,-1, 2)$.
Then $w_\Delta\geq\frac14$, that is, the Haar component has weight at least
$\frac14$.
\end{prop}

\begin{proof} The possible sets $R$ are
$\Delta, \emptyset, \{\alpha, -\alpha\}$.
In the case of $SL_n$ the roots are indexed by
$\{ij, 1\leq i, j\leq n, i\not=j\}$:  $\alpha_{ij}$ is defined by
$\alpha_{ij}(X)=X_{ii}-X_{jj}$. Consider $a=\diag(e^1, e^1, e^{-2})$.
Then $h_{KS}(\mu_{Haar}, a)=6$ (since $s_\alpha=1$ for all $\alpha$), $h_{KS}(\mu_\emptyset, a)=0$, 
$h_{KS}(\mu_{12}, a)=0$, $h_{KS}(\mu_{13}, a)\leq 3$, $h_{KS}(\mu_{23}, a)\leq 3$. Thus,
\begin{equation}\label{e:entropygame}
3\leq h_{KS}(\mu, a)\leq 3w_{13}+3w_{23}+6w_{\Delta}.
\end{equation}
This implies
$$w_{\Delta}-w_{12}\geq 1-(w_{\Delta}+w_{12}+w_{13}+w_{23})\geq 0.$$
By symmetry it follows that $w_{\Delta}\geq w_{13}$ and
$w_{\Delta}\geq w_{23}$.  Returning to \eqref{e:entropygame},
it follows that $3 \leq 12 w_{\Delta}$.

In fact, if $h_{KS}(\mu, a)\geq \left(\frac13+\eps\right) h_{KS}(\mu_{Haar}, a)$ for $a=e^X$, $X=\diag(2, -1,-1)$, $\diag(-1, 2-1)$, or $\diag( -1,-1, 2)$, then $w_\Delta\geq \frac32 \eps.$
\end{proof}

Putting together Theorem \ref{t:main2} and Proposition \ref{p:SL3} 
gives Theorem \ref{t:coroSL3}.

For $SL_4$ the analogue of Proposition \ref{p:SL3} is given
below.  Theorem \ref{t:coroSL4} is an immediate corollary.

\begin{prop}\label{p:SL4}
Let $G=SL_4(\IR)$, $\mu$ an $A$-invariant probability measure on
$\Gamma\backslash G$, such that
$h_{KS}(\mu, a)\geq \left(\frac12+\eps\right) h_{KS}(\mu_{Haar}, a)$ for
$a=e^X$, $X$ in the Weyl orbit of $\diag(3,-1,-1,-1)$.
Then $w_{\Delta}\geq 2\eps$. If $\eps=0$ and there is no Haar component, then each ergodic component is the Haar measure on a closed orbit of the group $\left( \begin{array}{cccc} 

* & * &*& 0 \\
* & * &*& 0 \\
* & * &*& 0 \\
 0& 0 &0& * \\
 
 \end{array}  \right)$ (or one of its 4 images under the Weyl group), and the components invariant by any of these 4 subgroups
have total weight $\frac14$.
\end{prop}

Theorem \ref{t:coroSL3} and its analogue for $G=SL_4$ apply to any lattice
$\Gamma$.  On the other hand for $G=SL_n$ 
with $n$ large some quotients $\Gamma\backslash G$ support ergodic invariant
measures of large entropy other than Haar measure,
so our entropy bound is not strong enough
to obtain a Haar component.  However, for some lattices $\Gamma$ there
are further restrictions on the set of ergodic components, so that
non-Haar measures have much smaller entropy.
This is the case where $\Gamma$ is a lattice associated to a divison
algebra.

We give here a quick outline of the construction, refering the reader
to~\cite{Tom00} and its references (or~\cite{Pla94}) for a detailed discussion.
Let $F$ be a central simple algebra of degree $n$ over $\IQ$ and assume that
$F$ \emph{splits} over $\IR$, that is that $F\otimes_{\IQ}\IR\simeq M_n(\R)$.
Next, let $\mathcal{O}\subset F$ be an \emph{order}, that is
a subring whose additive group is generated by a basis for $F$ over $\IQ$.
Finally, let $\mathcal{O}^1 \subset SL_n(\IR)$ denote the subgroup
of elements of $\mathcal{O}$ with determinant $1$ (``reduced norm $1$'').
Such $\mathcal{O}^1$ are in fact lattices; any lattice $\Gamma<SL_n(\IR)$
commensurable with some $\mathcal{O}^1$ is said to be of \emph{inner type}.
We simply say that they are \emph{associated} to the algebra $F$.
Our Theorem \ref{t:main2} applies when the lattice is co-compact,
which is the case if and only if $F$ is a division algebra.

We shall need the fact that those measure rigidity results of~\cite{EKL06}
which are stated specifically for $SL_n(\IZ)$ apply, in fact, to any lattice
of inner type, since the proof of Lemma 5.2 of that paper carries over to
the more general situation.  We give the easy argument here:

\begin{lem}\label{l:inner}
Let $\Gamma < SL_n(\R)$ be a lattice of inner type.
Then there is no $\gamma \in \Gamma$, diagonalizable in $\SL_n(\R)$,
such that $\pm 1$ are not eigenvalues of $\gamma$ and all eigenvalues of
$\gamma$ are simple except for precisely one which occurs with
multiplicity two.
\end{lem}

\begin{proof}
Say that $\Gamma$ is associated to the central simple algebra $F$,
and let $\mathcal{O}$ be an order in $F$ such that
$\Gamma \cap \mathcal{O}^1$ has finite index in $\Gamma$.

Assume by contradiction that there exists $\gamma$ as in the statement,
and choose $r$ so that $\gamma^r \in \mathcal{O}$.
Since $\mathcal{O}$ is a ring with
a finitely generated additive group, the Cayley-Hamilton Theorem shows that
$\gamma^r$ is integral over $\Z$.  It follows that every eigenvalue of
$\gamma^r$, hence of $\gamma$, is an algebraic integer. The fact that
$\det(\gamma) = 1$ now shows that the rational eigenvalues of $\gamma$
must be integral divisors of $1$, so by assumption all eigenvalues
of $\gamma$ are irrational.  Let $f(x)\in\IR[x]$ be the characteristic
polynomial of $\gamma$, when $\gamma$ is thought of as an element of
$SL_n(\IR)$.  We will show $f(x)\in\IQ[x]$.
Then the multiplicity the eigenvalues of $f$ would
be Galois invariant giving the desired contradiction.  For the last claim
extend scalars to $\IC$ and note that the usual proof that the reduced
trace and norm belong to $\IQ$ applies to the entire characteristic polynomial.
\end{proof}

\begin{prop}\label{p:inner}
Let $n \geq 3$ and let $t$ be the largest proper divisor of $n$.
Let $G=SL_n(\IR)$ and let $\Gamma<G$ be a lattice of inner type.
Let $\mu$ be an $A$-invariant probability measure on $\Gamma\backslash G$
such that
$h_{KS}(\mu, a)\geq \frac12 h_{KS}(\mu_{Haar}, a)$ for $a=e^X$, $X$ a Weyl
conjugate of $\diag(n-1, -1,\cdots, -1)$.
Then $w_\Delta\geq\frac{\frac{(n+1)}2-t}{n-t} > 0$.  In other words, $\mu$
must contain an ergodic component proprtional to Haar measure.
\end{prop}

Theorem \ref{t:coroinner} follows.

\begin{proof}
As above, let $\mu_x$ be an ergodic component of $\mu$ that has
positive entropy with respect to $e^X$.  By~\cite[Thm.\ 1.3]{EKL06}
(replacing Lemma 5.2 of that paper with Lemma \ref{l:inner} above)
$\mu_x$ must be \emph{algebraic}: there exists a closed subgroup
$H$ containing $A$,
and a closed orbit $zH$ in $\Gamma\backslash G$, such that $\mu_x$ is the
H-invariant measure on $zH$.  By~\cite{LW01} (the arguments are contained
in the proof of Lemma 4.1 and Lemma 6.2) and~\cite{Tom00}
(see Thm 1.2 and \S 4.2), $H$ must be reductive, and conjugate to
the connected component of $GL_k(\IR)^l\cap SL_n(\IR)$; where $n=kl$
and $GL_k(\IR)^l$ denotes the block-diagonal embedding of $l$ copies of
$GL_k(\IR)$ into $GL_n(\IR)$.

By the discussion following Proposition \ref{p:rigid} we see that for such
lattices $\Gamma$ the possible sets $R$ are obtained by partitioning $n$
into $l$ subsets $B_1, B_2, \ldots, B_l$ of equal size $k$, and letting
$$R=\{\alpha_{ij}, 1\leq i, j\leq n, \exists u \mbox{ such that } i\in B_u \mbox{ and } j\in B_u\}.$$

Consider $a=\diag(e^{n-1}, e^{-1}, \ldots, e^{-1})$.
Then $h_{KS}(\mu_{Haar}, a)=n(n-1)$, and for every subset $R$ defined as
above by a non-trivial partition, we have $h_{KS}(\mu_R)\leq n(t-1)$.
The inequality $h_{KS}(\mu, a)\geq \frac12 h_{KS}(\mu_{Haar}, a)$ now shows that
$$w_{\Delta}(n-1) + \sum_{R\neq\Delta} w_R (t-1) \geq \frac{n-1}2.$$
In other words we have
$$w_{\Delta}(n-1)+(1-w_\Delta) (t-1)\geq \frac{n-1}{2},$$
which is equivalent to the statement of the theorem.
\end{proof}

 %%%%%%%%%%%%%%%%%%%%%%%%%%%%%%%%%%%%
%%%%%%%%%%%%%%%%%%%%%%%%%%%%%%%%%%%%
%%%%%%%%%%%%%%%%%%%%%%%%%%%%%%%%%%%%

 \end{document}

%% file: def.7.tex
\newcommand{\nwc}{\newcommand}

\nwc{\mf}{\mathbf} %Latex (as in \bf not tilted math letters)
\nwc{\blds}{\boldsymbol} %Latex 
\nwc{\ml}{\mathcal} %Latex

\nwc{\lam}{\lambda}
\nwc{\del}{\delta}
\nwc{\Del}{\Delta}
\nwc{\Lam}{\Lambda}
\nwc{\elll}{\ell}
\nwc{\IA}{\mathbb{A}} %algebraic
\nwc{\IB}{\mathbb{B}} %ball
\nwc{\IC}{\mathbb{C}} %complex
\nwc{\ID}{\mathbb{D}} %Dedekind
\nwc{\IE}{\mathbb{E}} %Euklides
\nwc{\IF}{\mathbb{F}} %finite field
\nwc{\IG}{\mathbb{G}} %Gauss
\nwc{\IH}{\mathbb{H}} %Hilbert\N-subgroup
\nwc{\IN}{\mathbb{N}} %natural
\nwc{\IP}{\mathbb{P}} %prime
\nwc{\IQ}{\mathbb{Q}} %rational
\nwc{\IR}{\mathbb{R}} %real
\nwc{\IS}{\mathbb{S}} %sphere
\nwc{\IT}{\mathbb{T}} %torus
\nwc{\IZ}{\mathbb{Z}} %integers

\def\bbleft{{\mathchoice {[\mskip-3mu {[}} {[\mskip-3mu {[}}{[\mskip-4mu {[}}{[\mskip-5mu {[}}}}
\def\bbright{{\mathchoice {]\mskip-3mu {]}} {]\mskip-3mu {]}}{]\mskip-4mu {]}}{]\mskip-5mu {]}}}}
\nwc{\setK}{\bbleft 1,K \bbright}
\nwc{\setN}{\bbleft 1,\cN \bbright}
\nwc{\va}{{\bf a}}
\nwc{\vb}{{\bf b}}
\nwc{\vc}{{\bf c}}
\nwc{\vd}{{\bf d}}
\nwc{\ve}{{\bf e}}
\nwc{\vf}{{\bf f}}
\nwc{\vg}{{\bf g}}
\nwc{\vh}{{\bf h}}
\nwc{\vi}{{\bf i}}
\nwc{\vI}{{\bf I}}
\nwc{\vj}{{\bf j}}
\nwc{\vk}{{\bf k}}
\nwc{\vl}{{\bf l}}
\nwc{\vm}{{\bf m}}
\nwc{\vM}{{\bf M}}
\nwc{\vn}{{\bf n}}
\nwc{\vo}{{\it o}}
\nwc{\vp}{{\bf p}}
\nwc{\vq}{{\bf q}}
\nwc{\vr}{{\bf r}}
\nwc{\vs}{{\bf s}}
\nwc{\vt}{{\bf t}}
\nwc{\vu}{{\bf u}}
\nwc{\vv}{{\bf v}}
\nwc{\vw}{{\bf w}}
\nwc{\vx}{{\bf x}}
\nwc{\vy}{{\bf y}}
\nwc{\vz}{{\bf z}}
\nwc{\bal}{\blds{\alpha}}
\nwc{\bep}{\blds{\epsilon}}
\nwc{\barbep}{\overline{\blds{\epsilon}}}
\nwc{\bnu}{\blds{\nu}}
\nwc{\bmu}{\blds{\mu}}
\nwc{\bet}{\blds{\eta}}

\nwc{\bk}{\blds{k}}
\nwc{\bm}{\blds{m}}
\nwc{\bM}{\blds{M}}
\nwc{\bp}{\blds{p}}
\nwc{\bq}{\blds{q}}
\nwc{\bn}{\blds{n}}
\nwc{\bv}{\blds{v}}
\nwc{\bw}{\blds{w}}
\nwc{\bx}{\blds{x}}
\nwc{\bxi}{\blds{\xi}}
\nwc{\by}{\blds{y}}
\nwc{\bz}{\blds{z}}

\nwc{\cA}{\ml{A}}
\nwc{\cB}{\ml{B}}
\nwc{\cC}{\ml{C}}
\nwc{\cD}{\ml{D}}
\nwc{\cE}{\ml{E}}
\nwc{\cF}{\ml{F}}
\nwc{\cG}{\ml{G}}
\nwc{\cH}{\ml{H}}
\nwc{\cI}{\ml{I}}
\nwc{\cJ}{\ml{J}}
\nwc{\cK}{\ml{K}}
\nwc{\cL}{\ml{L}}
\nwc{\cM}{\ml{M}}
\nwc{\cN}{\ml{N}}
\nwc{\cO}{\ml{O}}
\nwc{\cP}{\ml{P}}
\nwc{\cQ}{\ml{Q}}
\nwc{\cR}{\ml{R}}
\nwc{\cS}{\ml{S}}
\nwc{\cT}{\ml{T}}
\nwc{\cU}{\ml{U}}
\nwc{\cV}{\ml{V}}
\nwc{\cW}{\ml{W}}
\nwc{\cX}{\ml{X}}
\nwc{\cY}{\ml{Y}}
\nwc{\cZ}{\ml{Z}}

\nwc{\fA}{\mathfrak{a}}
\nwc{\fB}{\mathfrak{b}}
\nwc{\fC}{\mathfrak{c}}
\nwc{\fD}{\mathfrak{d}}
\nwc{\fE}{\mathfrak{e}}
\nwc{\fF}{\mathfrak{f}}
\nwc{\fG}{\mathfrak{g}}
\nwc{\fH}{\mathfrak{h}}
\nwc{\fI}{\mathfrak{i}}
\nwc{\fJ}{\mathfrak{j}}
\nwc{\fK}{\mathfrak{k}}
\nwc{\fL}{\mathfrak{l}}
\nwc{\fM}{\mathfrak{m}}
\nwc{\fN}{\mathfrak{n}}
\nwc{\fO}{\mathfrak{o}}
\nwc{\fP}{\mathfrak{p}}
\nwc{\fQ}{\mathfrak{q}}
\nwc{\fR}{\mathfrak{r}}
\nwc{\fS}{\mathfrak{s}}
\nwc{\fT}{\mathfrak{t}}
\nwc{\fU}{\mathfrak{u}}
\nwc{\fV}{\mathfrak{v}}
\nwc{\fW}{\mathfrak{w}}
\nwc{\fX}{\mathfrak{x}}
\nwc{\fY}{\mathfrak{y}}
\nwc{\fZ}{\mathfrak{z}}

\newcommand{\lieg}{\mathfrak{g}}
\newcommand{\lien}{\mathfrak{n}}
\newcommand{\lienb}{\bar\lien}
\newcommand\Nf{N_\textrm{fast}}
\newcommand\Nbf{\overline{N}_\textrm{fast}}
\newcommand\lnf{\lien_\textrm{fast}}
\newcommand\lns{\lien_\textrm{slow}}
\newcommand\lnbf{\lienb_\textrm{fast}}
\newcommand\lnbs{\lienb_\textrm{slow}}

\nwc{\tA}{\widetilde{A}}
\nwc{\tB}{\widetilde{B}}
\nwc{\tE}{E^{\vareps}}
\nwc{\tk}{\tilde k}
\nwc{\tN}{\tilde N}
\nwc{\tP}{\widetilde{P}}
\nwc{\tQ}{\widetilde{Q}}
\nwc{\tR}{\widetilde{R}}
\nwc{\tV}{\widetilde{V}}
\nwc{\tW}{\widetilde{W}}
\nwc{\ty}{\tilde y}
\nwc{\teta}{\tilde \eta}
\nwc{\tdelta}{\tilde \delta}
\nwc{\tlambda}{\tilde \lambda}
\nwc{\ttheta}{\tilde \theta}
\nwc{\tvartheta}{\tilde \vartheta}
\nwc{\tPhi}{\widetilde \Phi}
\nwc{\tpsi}{\tilde \psi}
\nwc{\tmu}{\tilde \mu}

\nwc{\To}{\longrightarrow} %limits

\nwc{\ad}{\rm ad}
\nwc{\eps}{\epsilon}
\nwc{\ep}{\epsilon}
\nwc{\vareps}{\varepsilon}

\nwc{\bom}{\underline{\omega}}
\def\om{{\omega}}
\def\ep{\epsilon}

\def\diag{{\rm diag}}

\def\sq2{\sqrt{2}}

\def\t2{{\mathbb T}^2}
\def\s2{{\mathbb S}^2}

\def\R{\mathbb{R}}

\def\Z{\mathbb{Z}}
\def\C{\mathbb{C}}

\nwc{\lap}{\bigtriangleup}
\nwc{\rest}{\restriction}
\nwc{\Diff}{\operatorname{Diff}}
\nwc{\diam}{\operatorname{diam}}
\nwc{\Res}{\operatorname{Res}}
\nwc{\Spec}{\operatorname{Spec}}
\nwc{\Vol}{\operatorname{Vol}}
\nwc{\Op}{\operatorname{Op}}
\nwc{\supp}{\operatorname{supp}}
\nwc{\Span}{\operatorname{span}}

\nwc{\dia}{\varepsilon}
\nwc{\cut}{f}
\nwc{\qm}{u_\hbar}

\def\hto0{\xrightarrow{\hbar\to 0}}

\def\rto0{\xrightarrow{r\to 0}}

\providecommand{\abs}[1]{\lvert#1\rvert}
\providecommand{\norm}[1]{\lVert#1\rVert}

\nwc{\la}{\langle}
\nwc{\ra}{\rangle}
\nwc{\lp}{\left(}
\nwc{\rp}{\right)}

\nwc{\bequ}{\begin{equation}}
\nwc{\be}{\begin{equation}}
\nwc{\ben}{\begin{equation*}}
\nwc{\bea}{\begin{eqnarray}}
\nwc{\bean}{\begin{eqnarray*}}
\nwc{\bit}{\begin{itemize}}
\nwc{\bver}{\begin{verbatim}}

\nwc{\eequ}{\end{equation}}
\nwc{\ee}{\end{equation}}
\nwc{\een}{\end{equation*}}
\nwc{\eea}{\end{eqnarray}}
\nwc{\eean}{\end{eqnarray*}}
\nwc{\eit}{\end{itemize}}
\nwc{\ever}{\end{verbatim}}

\newcommand{\defeq}{\stackrel{\rm{def}}{=}}
\newcommand{\wl}{w_\ell}

\newcommand{\bN}{\overline{N}}
\newcommand{\bPz}{\bar{P}_0}
\newcommand{\bd}{\bar{d}}
\newcommand{\ars}{\fA^{*}}

\newcommand{\SL}{\textrm{SL}}
\newcommand{\bX}{\mathbf{X}}
\newcommand{\bY}{\mathbf{Y}}
\newcommand{\bS}{\mathbf{S}}

\newcommand\wrt{w.r.t.\ }
\DeclareMathOperator{\Jac}{Jac}
\DeclareMathOperator{\Ad}{Ad}